\documentclass[12pt]{amsart}
\usepackage{amsmath,amsfonts,amssymb} 
\usepackage{pdfpages}
\usepackage{amsbsy}
\usepackage{amscd}

\usepackage{amsthm}
\usepackage{bbold}
\usepackage{mathrsfs}
\usepackage{verbatim}
\usepackage [all]{xy}
\usepackage[top=1in,bottom=1in,left=1.25in,right=1.25in]{geometry}

\usepackage{tikz-cd}
\usepackage[pdftex]{hyperref}
\hypersetup{colorlinks,citecolor=blue,linktocpage,hyperindex=true,backref=true}

\newcommand{\hide}[1]{}

\theoremstyle{plain}
\newtheorem{thmi}{Theorem}
\newtheorem{cori}[thmi]{Corollary}
\newtheorem{thm}{Theorem}[section]
\newtheorem{prop}[thm]{Proposition}

\newtheorem{cor}[thm]{Corollary}
\newtheorem{lem}[thm]{Lemma}

\newtheorem*{Cartan-Dieudonne}{Cartan-Dieudonne theorem {\rm (\cite [Chapter I, Theorem 7.1]{Lam})}}
\newtheorem*{Local-Reduction-Lemma}{The Local Reduction Lemma}

\theoremstyle{definition}

\newtheorem{rem}[thm]{Remark}

\newtheorem{rem*}{Remark}

\theoremstyle{remark}

\newcommand{\HH}{{\mathbb H}}

\newcommand{\cF}{{\mathcal F}}
\newcommand{\CC}{{\mathbb C}}
\newcommand{\QQ}{{\mathbb Q}}
\newcommand{\RR}{{\mathbb R}}

\newcommand{\tM}{\widetilde{M}}
\newcommand{\PP}{{\mathbb P}}

\newcommand{\cT}{{\mathcal T}}

\newcommand{\ra}{\rightarrow}

\title{Twistor lines in the period domain of complex tori}

\author{Nikolay Buskin}

\address{Department of Mechanics and Mathematics, Novosibirsk State University, 1 Pirogova st., Novosibirsk, 
630090, Russia}

\email{nvbuskin@gmail.com}
\author{Elham Izadi}

\address{Department of Mathematics, University of California San Diego, 9500 Gilman Drive \# 0112, La Jolla, CA 92093-0112, USA}

\email{eizadi@math.ucsd.edu}
\begin{document}
\begin{abstract}
As in the case of irreducible holomorphic symplectic manifolds, the period domain $Compl$ of 
compact complex tori of even dimension $2n$ contains twistor lines. These are special $2$-spheres parametrizing complex tori whose complex structures arise from a given quaternionic structure. In analogy with the case of irreducible holomorphic symplectic manifolds, we show that the periods of any two complex tori can be joined by a {\em generic} chain of twistor lines. We also prove a criterion of twistor path connectivity of loci in $Compl$
where a fixed second cohomology class stays of Hodge type (1,1). 
Furthermore, we show that twistor lines are holomorphic submanifolds of $Compl$, of degree $2n$ in the Pl\"ucker embedding of $Compl$.
\end{abstract}
\maketitle
\tableofcontents
\section*{Introduction}

Let $M$ be a Riemannian manifold of real dimension $4m$ with metric $g$. Then $M$ is called {\it hyperk\"ahler} 
with  respect to $g$ (see \cite[p. 548]{Hitchin})
if there exist complex structures $I$, $J$ and $K$ on $M$, such that 
$I,J,K$ are covariantly constant and 
are isometries of the tangent bundle $TM$ with respect to $g$, satisfying
the quaternionic relations
\[
I^2=J^2=K^2=-1, \quad IJ=-JI=K. 
\]
We call the ordered triple $(I,J,K)$ {\it a hyperk\"ahler
structure on $M$ compatible with $g$}.

A hyperk\"ahler structure $(I,J,K)$ gives rise to a sphere $S^2$ of complex structures on $M$:
$$S^2=\{aI+bJ+cK| a^2+b^2+c^2=1\}.$$

We call the family $\mathcal M=\{(M,\lambda)| \lambda \in S^2\}  \rightarrow S^2$ a {\it twistor family over the twistor sphere $S^2$}. 
The family $\mathcal M$ can be endowed with a complex structure,
so that it becomes a complex manifold and the fiber $\mathcal M_\lambda$ is biholomorphic to
the complex manifold $(M,\lambda)$, see \cite[p. 554]{Hitchin}.
For every $\lambda=aI+bJ+cK \in S^2$, the closed alternating form $g(\lambda\cdot,\cdot)$
determines a K\"ahler class in $H^{1,1}((M,\lambda),\RR)$.

The known examples of compact hyperk\"ahler manifolds are even-dimensional complex tori and irreducible holomorphic symplectic manifolds ({\it IHS manifolds}).
We recall that an IHS manifold is a simply connected compact K\"ahler manifold $M$  with $H^0(M,{\Omega}_M^2)$ generated by
an everywhere non-degenerate holomorphic 2-form.
Examples of IHS manifolds include  $K3$ surfaces and, more generally, Hilbert schemes of points on $K3$
surfaces, generalized Kummer varieties.

For IHS manifolds and complex tori there exist well-defined period domains, carrying the structure of a complex manifold. Every twistor family $\mathcal M$
determines an embedding of the base $S^2$ into the corresponding period domain 
as a 1-dimensional complex submanifold (for IHS manifolds this is known and for complex tori we give a proof of this in Theorem 
\ref{Theorem-tw-path-conn}). The image of such an embedding is called a {\it twistor line}.
The period of a hyperk\"ahler manifold is called {\it generic}, if the corresponding manifold has trivial N\'eron-Severi group. 
A path of twistor lines is an ordered sequence $S_1, \ldots , S_m$ of twistor spheres such that $S_i \cap S_{i+1}$ is non-empty if $1\leq i\leq m-1$. Such a path is called {\em generic}, if the periods at intersections of successive lines in the path are generic.

In the case of IHS manifolds it is known that any two periods can be connected
by a generic path of twistor lines
 (see \cite{Bourbaki}, which is an exposition of Verbitsky \cite{Verb-Torelli}).  In \cite{Verb-Torelli} generic twistor path connectivity was used to prove surjectivity of the corresponding 
period mapping, which was a part of the Torelli theorem for IHS manifolds proved there. Another application was given in \cite{Buskin}, where generic twistor path connectivity, together with a result of Verbitsky cited below, was used to prove, via deformations of sheaves, that every rational Hodge isometry between two $K3$-surfaces is algebraic. 

The period domain $Compl$ for complex tori of dimension $2n$ is a real analytic open subset of the complex
Grassmanian $Gr(2n,4n)$, see the exact definition in Section 
\ref{Space-of-twistor-lines:-dimension-count}. It has two connected components $Compl^+$ and $Compl^-$.
Our first main result is

\begin{thmi}
\label{Theorem-tw-path-conn}\begin{enumerate}
\item Any twistor sphere on a complex torus naturally embeds into $Compl$ as a complex 1-dimensional submanifold. The degree of twistor lines in $Gr(2n,4n)$
with respect  to the Pl\"ucker embedding is $2n$.
\item In each of the two connected components of $Compl$ any two periods can be connected by a generic path
of twistor lines.
\end{enumerate}
\end{thmi}

For an alternating 2-form $\Omega$, we let $Compl_\Omega$
be the locus of all periods $I \in Compl$ such that the form $\Omega$ determines a cohomology class
of Hodge type (1,1) in the second cohomology of tori with period (i.e., complex structure, see Section \ref{Space-of-twistor-lines:-dimension-count}) $I$. 

We let  
$G_\Omega \subset G \cong GL_{4n} (\RR)$ be the group of automorphisms of $\Omega$ with $G^0_\Omega$ its connected component of the identity. 
For $I\in Compl_\Omega$ we define a hermitian form 
$h_I$ by setting $h_I(u,v)=\Omega(u,Iv)-i\Omega(u,v)$.
We let $(n_+,n_-,n_0)$ be the signature of $h_I$, 
$n_0$ depends on $\Omega$ only and is the same for all $I \in Compl_\Omega$.  
We let $U(n_+,n_-,n_0)\subset G^0_\Omega$ 
be the automorphism group of $h_I$.

We set $Compl_\Omega^{\pm}=Compl_\Omega \cap Compl^{\pm}$. Our second main result is 

\begin{thmi} \label{mainthm2}
\label{Tw-path-conn-Compl-Omega}
\begin{enumerate}
\item The locus $Compl_\Omega^{\pm}$ has $2n+1-n_0$ connected components, 
indexed by  the signature $(n_+,n_-,n_0)$ 
of $h_I$ for all periods $I$ in the corresponding component.

The connected component
of $Compl_\Omega^{\pm}$, where $h_I$ has signature 
$(n_+,n_-,n_0)$, is naturally diffeomorphic to the homogeneous manifold $G^0_\Omega/U(n_+,n_-,n_0)$
and is a  smooth complex submanifold in $Compl$. 

\item If $n_0$ is even, there is precisely one connected component of $Compl_\Omega^{\pm}$, 
that contains twistor lines. This component corresponds to the signature $\left(n-\frac{n_0}{2},n-\frac{n_0}{2},n_0\right)$ 
and is twistor path connected. 
If $n _0$ is odd, there are no connected components of $Compl_\Omega^{\pm}$
containing a twistor line.
\end{enumerate}
\end{thmi}

An immediate consequence of this theorem is that K\"ahler classes on complex tori do not stay of Hodge type in twistor families.
An analogous result for $K3$ surfaces is well known. 

In fact, combining Theorem \ref{mainthm2} with the result of Verbitsky (see below) we obtain the following

\begin{cori}
Let $M$ be a hyperk\"ahler manifold and $\omega = g (I \cdot, \cdot)$ a K\"ahler class
on $M$ associated to a complex structure $I$ of the twistor family. 
If an $\omega$-slope-polystable bundle over $M$ extends to the twistor family $\mathcal M$, then, either its first Chern class is zero, or the hermitian form associated to its first Chern class has signature of the form $\left(n-\frac{n_0}{2},n-\frac{n_0}{2},n_0\right)$.
\end{cori}

We briefly recall Verbitsky's result.
Assume $M$ is a hyperk\"ahler manifold (not necessarily simply connected) with Riemannian metric $g(\cdot,\cdot)$ and a fixed complex structure $I$, let us denote again by $\omega$ the K\"ahler class on $M$ represented by the 
form $\omega(\cdot,\cdot)=g(I\cdot,\cdot)$. 
Then, 
by definition, we have a sphere of complex structures on $M$, the corresponding twistor family $\mathcal M \to S^2$, and a K\"ahler class 
represented by the 
form $\omega_\lambda(\cdot,\cdot)=g(\lambda \cdot,\cdot)$ on the fiber
$M_\lambda=(M,\lambda)$ for each $\lambda \in S^2$, such that $M_I=M$ and $\omega_I=\omega$. By definition, the class $\omega_\lambda$
(considered up to multiplication by a positive scalar) is {\it the K\"ahler class on $M_\lambda$}. 

Recall that a vector bundle on $M$ is called $\omega$-slope-polystable if it is isomorphic to a direct sum of $\omega$-slope-stable bundles with equal slopes.
The following theorem was proved in  \cite[Thm. 3.17, Thm. 3.19]{Verb-book}. 
\vskip5pt

{\bf Theorem}. {\it Let $F$ be an $\omega$-slope-polystable vector bundle over a hyperk\"ahler manifold $M$. If the Chern classes $c_1(F)$ and $c_2(F)$ remain of Hodge type for all complex structures $\lambda$ on $M$ belonging to the sphere $S^2$, then the bundle $F$ extends to a vector bundle $\cF$ over $\mathcal M$. Furthermore, for all $\lambda \in S^2$, the restriction $\cF |_{M_\lambda}$ is an $\omega_\lambda$-slope-polystable bundle}.
\vskip5pt

A purely
geometric motivation behind the study of the connectivity, besides the application to deforming sheaves, 
is discussed in Remark \ref{sub-riemannian-remark} below. 

Let us say few words on how proving the connectivity for the period
domain of complex tori differs from that for the period domains of IHS manifolds. 
For 
IHS manifolds,
the proof of connectivity relies on the realization of the period domain
as the Grassmanian of oriented positive real 2-planes in the second cohomology, where positivity
is with respect to the Beauville-Bogomolov bilinear form, again see \cite{Bourbaki}. 
This bilinear form provides a 
very convenient tool for investigating the local topology of this period domain with respect to the problem of twistor path connectivity. 

 For complex tori, however, we 
do not have such a realization of their period domain and cannot use a similar argument. Here, instead,
we need to use the (less refined) structure of  the period domain of complex tori as a homogeneous space (which, certainly, the period
domain of an IHS manifold is, as well). This homogeneous structure allows
us to proceed with proving the twistor path connectivity in steps that are, in their broad strokes,  parallel to the steps
of the proof of the twistor path connectivity for the period domains of IHS manifolds.

\begin{rem*}
As shown by Beauville \cite{Beauvillec1=0} (using results of Cheeger and Gromoll), a general compact hyperk\"ahler manifold $M$ has a finite \'etale cover $\tM$ isomorphic to the product of a complex torus and a finite number of irreducible hyperk\"ahler manifolds. Since the irreducible hyperk\"ahler manifolds are simply connected, one can easily see that the N\'eron-Severi group of $\tM$ is isomorphic to the direct sum of the N\'eron-Severi groups of its factors. Twistor families for $M$ give rise to twistor families for $\tM$ and its factors. Vice-versa, twistor families for (any of) the factors of $\tM$ give rise, in a generally non-unique way, to 
twistor families for $\tM$ and $M$. One can then deduce the 
generic twistor path connectivity of the moduli space of $M$ from the generic twistor path connectivity of the moduli spaces of complex tori and 
those of irreducible hyperk\"ahler manifolds.
\end{rem*}

\begin{rem*}
\label{sub-riemannian-remark}
There is a relation between twistor path connectivity and rational connectedness, that is, the
connectivity of points of a complex manifold by chains of rational curves (for the latter see, for example, \cite{Kollar}).
The  Grassmanian $Gr(2n,4n)$ being
a rational variety ($Gr(2n,4n) \overset{\sim}{\dashrightarrow} \PP^{4n^2}$), is certainly  rationally connected. 
However, rational connectedness is  a weaker property than twistor path connectivity. 
Indeed, the variety of lines in $\PP^{4n^2}$, passing through a
fixed point, has complex dimension $4n^2-1$ (and the dimension of the variety
of rational curves of degree $d>1$ in $\PP^{4n^2}$, passing through a fixed point, is even larger), thus its real dimension is $8n^2-2$. 
On the other hand,
by Corollary \ref{Corollary-M-I-dimension}, the real dimension 
of the space of all twistor lines, passing through a fixed point in the period domain,
is $4n^2-1$. Thus, through a given point, there are ``half as many'' twistor lines as general rational curves,
and the problem of twistor path connectivity may be roughly considered as a ``sub-Riemannian''
version of rational connectedness problem.
\end{rem*}

The plan of the paper is as follows.

In Section \ref{Space-of-twistor-lines:-dimension-count} we describe our basic set-up, the complex-analytic structure
of $Compl$ considered as a real-analytic submanifold in $End(V_\RR)\cong End(\RR^{4n})$ and show that the twistor spheres $S^2 \subset Compl$ are complex submanifolds (Corollary \ref{corSanalytic}). We define the 
union $C_I$
of all twistor spheres passing through a given period $I$ and show that $G_I$
 acts transitively on the set of twistor spheres containing $I$. 
The sets $C_I$, which are real-analytic subsets in $Compl$, will serve as building blocks for constructing twistor paths.

 In Section \ref{Twistor-path-connectivity-of-Compl} we provide an argument, illustrated by a picture, that the set of periods reachable from a given one $I$ by means of all possible triples of consecutive twistor spheres contains an open neighborhood
of the initial period. Then, the connectedness of the period domain allows us to conclude that any two periods can be connected by a path of twistor lines.
The three spheres argument is essentially due to the transversality formulated in its most general form in Proposition \ref{Proposition-general-transversality},
and it is somewhat analogous to the ``three lines argument'' in \cite[Prop. 3.7]{Bourbaki}. 

In Section \ref{Connectivity-by-generic-twistor-paths}
we prove the generic connectivity part of Theorem \ref{Theorem-tw-path-conn}.
The idea of the proof is to show that the space of triples of consecutive
twistor spheres connecting a fixed pair of periods is not the union of
its two subspaces for which the first or, respectively, the second, of the two joint points 
belongs to the locus of tori with nontrivial N\'eron-Severi group in the period domain. 
Again, the transversality, stated in Proposition \ref{Proposition-general-transversality},
constitutes the main tool for proving generic connectivity.

In Section \ref{Twistor-spheres-are-complex-submanifolds} we prove that the degree of twistor lines in $Gr(2n,4n)$
with respect  to the Pl\"ucker embedding is $2n$. Here we use the fact that the group $G$ acts transitively on the set of all twistor lines,
 preserving the degree, together with an explicit computation on an explicit example.

In Section \ref{Twistor-path-connectivity-of-Compl-Omega} we prove the statement of Theorem \ref{Tw-path-conn-Compl-Omega}, the 
proof of connectivity is again based on the above mentioned ``three lines argument''. 

The authors acknowledge debts of gratitude: to Eyal Markman for suggesting the problem of twistor path connectivity
of the period domain, and to the referee for the content of Remarks \ref{RemPlucker}, \ref{remreferee}, and pointing out an error in the original calculation of the degree of twistor lines. The authors are indebted to Eyal Markman and the referee for many useful comments that helped improve the exposition of the paper.

\section{The space of twistor spheres}
\label{Space-of-twistor-lines:-dimension-count}
\subsection{} Let $A$ be a complex torus of dimension $2n$. Denote 
by $V_\RR$ the real tangent space $T_{\RR,0}A$ and
by $V$ the complex tangent space $T_{\CC,0}A\subset T_{\RR,0}A \otimes \CC$,
so that $\dim_\RR V_\RR = 2\dim_\CC V=4n$.

The points of $Compl$ are $2n$-dimensional complex planes, realizing the real weight 1 Hodge structures on the complex
$4n$-dimensional vector space $V_\CC:=V_\RR \otimes \CC$.  The open subset $Compl$ of $Gr(2n,4n)$ consists of those $2n$-planes
in $V_\CC$ that do not intersect the real subspace $V_\RR \subset V_\CC$.
Explicitly, a complex structure $I \colon V_\RR \rightarrow V_\RR$ corresponds to 
the $2n$-plane $(\mathbb{1}-iI)V_\RR \in Gr(2n,4n)=Gr(2n,V_\CC)$ where $\mathbb{1}$ denotes the identity map.
 As a homogeneous space, $Compl$ is the orbit of $I$ under the conjugation action of $G := GL(V_\RR)$:
$Compl\cong G/G_I$, where $G_I\cong GL_{2n}(\CC)$ is the stabilizer of $I$. This orbit
is endowed with a complex manifold structure such that the above embedding $I\mapsto (\mathbb{1}-iI)V_\RR \in Gr(2n,4n)$ is biholomorphic, see \cite[p. 31]{BirkenhakeLange99} and Proposition \ref{propCompl}.
 The period domain
$Compl$ consists of two connected components $Compl^+$ and 
$Compl^-$, corresponding to the components $GL^+(V_\RR)$
and $GL^-(V_\RR)$ of $G$.

Assume that
$J\colon V_\RR \rightarrow V_\RR$ is a complex structure anticommuting with $I$. 
Then $I$ and $J$ determine a twistor sphere $$S(I,J) :=\{aI+bJ+cK|a^2+b^2+c^2=1\},$$ where $K=IJ$.
In general, for two complex structures $I_1,I_2$, not necessarily anticommuting, such that $I_1 \neq \pm I_2$, 
and such that they are contained in the same twistor sphere $S$, we will also denote
this sphere by $S(I_1,I_2)$.
Our notation is justifed by the 
following lemma, whose proof is an exercise that we leave to the reader.

\begin{lem}\label{lempair}
Every twistor sphere $S$ is uniquely determined by any pair of non-proportional complex structures $I_1,I_2 \in S$.
\end{lem}

\subsection{} \label{subsecbasis} Let $J$ be a complex structure that anti-commutes with $I$. Then $V_\RR$ splits, in a non-unique way, as a direct sum of 4-dimensional subspaces of the form $\langle v,Iv,Jv,IJv\rangle$
for nonzero vectors $v\in V_\RR$, and the union of the specified bases of the 4-subspaces
forms a basis of $V_\RR$. In this basis the matrix of $J$ has a block-diagonal form with the following 4$\times$4 blocks on the diagonal
$$\left(\begin{array}{cc|cc} 0 & 0 &   -1 & 0\\ 0 & 0 &   0 & 1 \\
\hline
 1 & 0 & 0 & 0 \\ 0 & -1 & 0 & 0\end{array}\right).$$

\begin{prop}
\label{prop-hyperkahler}
Given a triple of complex structures $(I,J,K)$
on $A$ satisfying the quaternionic identities, there exist a (non-unique) metric $g$ on $A$ such that $(I,J,K)$ is a hyperk\"ahler structure with respect to $g$.
\end{prop}

\begin{proof}
Choose a basis of $V_\RR$ as in Paragraph \ref{subsecbasis} 
and define a metric $g(\cdot,\cdot)$ on $V_\RR$ by declaring this basis
to be orthonormal. Then $I,J$ and $K$ are isometries with respect to $g(\cdot,\cdot)$ and $g$ is K\"ahler with respect to all three complex structures. 
\end{proof}

\subsection{} By the definition of $Compl$, the group $G$ acts transitively on it: 
$$g \in G \colon J \mapsto \: ^g\! J= gJg^{-1}.$$ In particular, $G$ acts on the set of all twistor spheres 
$S(I,J)$ in $Compl$:
\[
g\cdot S(I,J)=S(\,^g\! I,\,^g\! J).
\]
For $g \in G_I$ we have $g \cdot S(I,J)=S(I,\,^g\!J)$.
We have

\begin{prop}
\label{transitivity}
The group $G_I$ acts transitively on the set $N_I$ of complex structures anticommuting with $I$.
\end{prop}

\begin{proof}
Let $J$ be a complex structure that anti-commutes with $I$.
The group $G_I \cong GL(V) < GL^+(V_\RR)=GL^+_{4n}(\RR)$ acts transitively on the set of bases as in Paragraph \ref{subsecbasis}, hence also on the set of $J$ anti-commuting with $I$. 
\end{proof}

\subsection{} Therefore, given a complex structure $J\in N_I$, $N_I = G_I \cdot J \cong G_I / G_{I,J}$ is the orbit of $J$ under $G_I$, where $G_{I,J}$ is the stabilizer group of $J$ in $G_I$.
Since $G_{I,J}$ is the subgroup of elements of $G_I = GL(V)$ commuting with $J$, that is, preserving the quaternionic structure
on $V_\RR$ determined by $I$ and $J$, we have $G_{I,J} \cong GL(V, \mathbb{H})$ which we will also
denote by $G_\HH$. So $N_I \cong GL(V)/GL(V, \mathbb{H})$ and we deduce

\begin{cor}\label{corNIdim}
The set $N_I$ is a real submanifold of $Compl$ of  dimension $4n^2$.
\end{cor}
\begin{proof}
The dimension of the orbit as a complex manifold is 
$\dim_\CC\, GL(V)-\dim_\CC\, GL(V, \mathbb{H})=(2n)^2-2n^2=2n^2$. 
The real dimension is thus
equal to $4n^2$.
\end{proof}


\subsection{}\label{subsecSMN} Let $S=S(I,J)$ for $J\in N_I$ be a twistor sphere.
Define $G_{I,S}\subset G_I$ to be the stabilizer of $S$ as a set, i.e., the set of elements $g$ of $G_I$ such that $g \cdot S \subset S$. For any $g\in G_{I,S}$, the complex
structure $^g\! J\in S$ also anticommutes with $I$, so $^g\! J$ is of the form
$aJ+bK, a^2+b^2=1$. Setting $a=\cos t, b=\sin t$ we have $aJ+bK=
e^{\frac{tI}{2}}Je^{-\frac{tI}{2}}$, where $e^{sI} = \cos s \: \mathbb{1} + \sin s \: I\in G_I$ realizes,
via the conjugation action,  
the rotations of $S$ around $\{\pm I\}$. Conversely, if $g\in G_I$ and $^g\!J \in S$, then $g \in G_{I,S}$.
The set of $g \in G_{I,S}$ such that $^g\!J=J$ is the quaternionic subgroup
$G_{I,J}=G_{\mathbb H} \subset G_{I,S}$. 
Explicitly, we have
$G_{I,S}=\langle e^{tI}, t\in \RR\rangle \times G_{\HH}$, where $\langle e^{tI}, t\in \RR\rangle \cong S^1$ (which is a subgroup of the center of $G_I$).
This
tells us, in particular, that $\dim_\RR\, G_{I,S}=\dim_\RR \,G_{\mathbb H}+1=4n^2+1$.

Let $M_I$ be the set of all twistor spheres in $Compl$ containing $I$. The natural map $N_I \ra M_I$ identifies two complex structures $J_1$ and $J_2$ whenever they belong to the same twistor sphere through $I$, i.e., $S(I, J_1) = S(I, J_2)$. More precisely, they belong to the great circle in $S := S(I, J_1)$ consisting of elements anticommuting with $I$. 
Hence, for the $S^1$-action $J \in N_I \mapsto e^{tI}J=e^{\frac{tI}{2}}Je^{-\frac{tI}{2}}$ on $N_I$ defined above, we have $N_I / S^1 = M_I$. Therefore Corollary \ref{corNIdim} immediately implies

\begin{cor}
\label{Corollary-M-I-dimension}
The set $M_I$ is a real manifold of dimension $4n^2 -1$.
\end{cor}

\subsection{The twistor cone of $I$} Define the set $C_I := \bigcup_{S\in M_I}S \subset Compl$ as the union of all twistor spheres containing $I$. 
All spheres in this union contain the complex structures $I$ and $-I$.
We will sometimes refer to the set $C_I$ as a (twistor) cone. Proposition \ref{transitivity} immediately implies
\begin{cor}\label{corMItrans}
The group $G_I$ acts transitively on $M_I\cong G_I/G_{I,S}$ so that
$C_I=\bigcup_{g\in G_I}g \cdot S(I,J)$.
\end{cor}
We will
give an explicit local parametrization of $C_I$ in  the next section and prove that 
the cone $C_I$ is a real-analytic subset of $Compl$ of dimension $4n^2+1$ (Proposition \ref{Prop-cone-dimension}).

\subsection{} We now 
describe the complex structure on the tangent bundle of the orbit
$Compl = G \cdot I$. Then we will see that the tangent bundle $TS^2$ of an arbitrary twistor sphere 
$S^2 \subset Compl$
is a subbundle of the restricted tangent bundle $TCompl|_{S^2}$, invariant under the complex structure of $TCompl|_{S^2}$.
This will imply the well-known fact that the twistor sphere $S^2$ is a complex submanifold in $Compl$.

\begin{prop}\label{propCompl}  The submanifold $Compl\subset End(V_\RR)$ is a complex manifold. Its complex structure $l_I$ is given by left multiplication by $I$ on $T_ICompl \subset End (V_\RR)$.
\end{prop}

The complex structure of $Compl$ is induced by that of $Gr(2n,4n)$ 
via the embedding $I\mapsto (\mathbb{1}-iI)V_\RR$. The proof of Proposition \ref{propCompl} is a technical exercise that we leave to the reader. 
\begin{proof}
Denoting the differential of the embedding $Compl\ni I\mapsto (\mathbb{1}-iI)V_\RR \in Gr(2n,4n)$
by $\varphi$ we have the following commutative diagram
\begin{equation*}
\begin{tikzcd}[row sep=huge]
T_ICompl\ni X \arrow[r,"\varphi"] \arrow[d,swap,"l_I"] & \left.\frac{d}{dt}\right|_{t=0}(\mathbb{1}-i\,{}^{e^{tX}}\!\!I)V_\RR\in T_{(\mathbb{1}-iI)V_\RR}Gr(2n,4n)
  \arrow[d,swap,"i\times"]\\ 
T_ICompl\ni Y\arrow[r,"\varphi"] & \left.\frac{d}{dt}\right|_{t=0}(\mathbb{1}-i\,{}^{e^{tY}}\!\!I)V_\RR\in T_{(\mathbb{1}-iI)V_\RR}Gr(2n,4n)
\end{tikzcd}
\end{equation*}
where $l_I$ denotes the  complex structure operator on $T_ICompl$ 
and `$i\times$' denotes the multiplication
by $i$ on $$\CC^{4n^2}\cong Hom( (\mathbb{1}-iI)V_\RR,V_\CC/ (\mathbb{1}-iI)V_\RR)=T_{(\mathbb{1}-iI)V_\RR}Gr(2n,4n),$$
so that $\varphi\circ l_I=i\varphi$.

We note that $T_ICompl\cong T_eG/T_eG_I$ and, as $T_eG_I$ consists of all operators in $End(V_\RR)$
commuting with $I$, the tangent space $T_ICompl$ can be identified with the subspace of operators in $End(V_\RR)$
anticommuting with $I$. Indeed, every operator $X\in End(V_\RR)$
can be written as a sum of an operator anticommuting with $I$ and an operator commuting with $I$, $X=\frac{1}{2}(X-X^I)+\frac{1}{2}(X+X^I)$, where $X^I=IXI^{-1}$. This allows us to immediately assume that $X$ and $Y$ in the above diagram anticommute with $I$.

Now we evaluate 
$$\varphi(X)= \left.\frac{d}{dt}\right|_{t=0}(\mathbb{1}-i\,{}^{e^{tX}}\!\!I)V_\RR=
\left\{ (\mathbb{1}-iI)v\mapsto -i(XI-IX)v=-2iXIv\,|\, v \in V_\RR\right\},$$
which, after multiplying by $i$ becomes $i\varphi(X)=\left\{ (\mathbb{1}-iI)v\mapsto 2XIv\,|\, v \in V_\RR\right\}$
(here we slightly abuse notation by writing, instead of the actual $\varphi(X),i\varphi(X)$, their representatives in $Hom((\mathbb{1}-iI)V_\RR,V_\CC)$). 

Now, considering $\varphi(Y)=\left\{ (\mathbb{1}-iI)v\mapsto -2iYIv=-2YiIv\,|\, v \in V_\RR\right\}$ as a vector in $Hom( (\mathbb{1}-iI)V_\RR,V_\CC/ (\mathbb{1}-iI)V_\RR)$,
we may write $\varphi(Y)=\left\{ (\mathbb{1}-iI)v\mapsto -2Yv\,|\, v \in V_\RR\right\}$. 
In order to have the equality  $\varphi\circ l_I=i\varphi$, setting $Y=l_I(X)$
we write $$\varphi(Y)=i\varphi(X)=\left\{ (\mathbb{1}-iI)v\mapsto 2XIv=-2IXv\,|\, v \in V_\RR\right\}.$$
In order for the latter equality to be true it is necessary and sufficient that $Y=IX$, that is, the map $l_I$ is the left multiplication by $I$ on
 $T_ICompl$.
\end{proof}

\begin{cor}\label{corSanalytic}
The twistor spheres $S^2 \subset Compl$ are complex submanifolds.
\end{cor}
\begin{proof}
The proof is based on the simple observation that the tangent space of $S^2=S(I,J)$ 
at the point $I$, 
for $I,J,K=IJ$ satisfying the quaternionic identities,  is the 2-plane $\langle J,K\rangle_\RR \subset T_ICompl$ 
and this plane is obviously invariant under left multiplication by $I$. Thus, $TS^2$ is a complex subbundle of $TCompl|_{S^2}$ and
thus  $S^2 \subset Compl$ is a complex submanifold.
\end{proof}

\begin{rem} 
\label{RemPlucker}
As was pointed out to us by the referee, there is an alternative proof of Corollary \ref{corSanalytic} 
that follows from considering $S\subset Compl$ as a subset in $Gr(2n,V_\CC)$. 
 Namely, denoting by $\mathbb{H}$ the algebra of quaternions and fixing
a representation $\mathbb{H}\to End(V_\RR)$ defined by $S=S(I,J)$, we obtain a structure of an $\mathbb{H}$-module on  $V_\RR$.
This $\mathbb{H}$-module is of the form $\mathbb{H}\otimes V^{\prime}$ for an $n$-dimensional  $\RR$-vector space $V^{\prime}$.
The eigenspace $V^{1,0}\subset V_\CC$ for a complex structure induced by the action of $\mathbb{H}$ is of the form 
$\mathbb{H}^{1,0}\otimes V^{\prime}_\CC $, where $V^{\prime}_\CC=V^{\prime}\otimes \CC$ and $\mathbb{H}^{1,0}$
is the corresponding eigenspace in $\mathbb{H}\otimes \CC$. Taking the tensor product with $V^{\prime}_\CC$
defines a complex analytic embedding $i\colon Gr(2,\mathbb{H}\otimes \CC) \hookrightarrow Gr(2n,V_\CC)$. 
Thus, every twistor line  in $Gr(2n,V_\CC)$ is the image of a twistor line in $Gr(2,\mathbb{H}\otimes \CC)$ under some embedding $i$ as above.
Now, the twistor lines in the quadric (under the Pl\"ucker embedding) $Gr(2,\mathbb{H}\otimes \CC)$ are known to be 
obtained as linear subspace sections, thus they are complex analytic submanifolds. Hence, our $S\subset Gr(2n,4n)$ is a complex analytic submanifold.
\end{rem}

\section{Twistor path connectivity of $Compl$}
\label{Twistor-path-connectivity-of-Compl}
The main result of this section is Theorem \ref{Theorem-connectivity}. 
Before proving it we need to introduce a certain mapping and prove an important technical result 
 about it (Proposition \ref {Injectivity-proposition}).

\subsection{} \label{subsecPhi} Let $I,J,K$ be a triple of complex structures belonging to a twistor sphere $S$.
 Consider the smooth mapping
 \[
 \begin{array}{cccc}
 \Phi \colon & G_J \times G_K & \longrightarrow & Compl,\\
& (g_1,g_2) &\longmapsto & ^{g_1g_2}\! I,
\end{array}
\]
where, as before, the action on $Compl$ is by conjugation: $g\cdot I=\,^g\!I=gIg^{-1}$. The mapping 
$\Phi$ clearly sends $G_\HH \times G_\HH$ to $I$, so that
 its differential $d\Phi_{(e,e)}$ factors through 
$$\widetilde{d\Phi_{(e,e)}}: T_eG_J/T_eG_{\HH} \oplus T_eG_K/T_eG_{\HH} \rightarrow T_ICompl.$$
\begin{prop}
\label{Injectivity-proposition}
Suppose $I,J,K$ is a quaternionic triple. The mapping 
$$\widetilde{d\Phi_{(e,e)}}: T_eG_J/T_eG_{\HH} \oplus T_eG_K/T_eG_{\HH} \rightarrow T_ICompl$$
is an isomorphism.
\end{prop}
Before proving Proposition \ref{Injectivity-proposition} we make the following  useful observation. For a vector $X\in T_eG\cong End\,V_\RR$
and $I\in Compl$ we introduce the notation $X^I=I^{-1}XI=IXI^{-1}\in T_eG$. 
Every vector $X\in T_eG$ can be uniquely decomposed into the sum
of its $I$-commuting and $I$-anticommuting components 
$$X=\frac{1}{2}(X+X^I)+\frac{1}{2}(X-X^I),$$  
so that interpreting the subspace $T_eG_I\subset T_eG$
as the subspace $\{Y\in T_eG\,|\,YI=IY\}$ of $I$-commuting vectors in $T_eG$, we get the natural isomorphism 
\begin{equation}
\label{Tangent-space}
T_ICompl\cong T_eG/T_eG_I\cong \{Y\in T_eG\,|\,YI=-IY\},
\end{equation}
of $T_ICompl$ with the subspace of $I$-anticommuting vectors in $T_eG$. Similarly we may write 
$$T_eG_J/T_eG_{\HH} \cong \{Y\in T_eG\,|\, YI=-IY,YJ=JY\}$$ and  
$$T_eG_K/T_eG_{\HH}\cong \{Y\in T_eG\,|\, YI=-IY,YK=KY\}.$$ 
\begin{proof}[Proof of Proposition \ref{Injectivity-proposition}]
By the definition of $\widetilde{d\Phi_{(e,e)}}$, its restrictions to the above direct summands
are injective. Let us show that it is injective on the direct sum. 
Consider $X\in T_eG_J, Y\in T_eG_K$ and the vector
$\widetilde{d\Phi_{(e,e)}}(X+T_eG_\HH, Y+T_eG_\HH)$, which is
 $$d\Phi_{(e,e)}(X+Y)=\left.\frac{d}{dt}\right|_{t=0}(e^{tX}e^{tY}\cdot I) = (X+Y)I-I(X+Y)\in T_ICompl.$$ Assume that this vector is zero,
that is, $X+Y$ commutes with $I$:
\begin{equation}
\label{fail-of-injectivity}
I(X+Y)=(X+Y)I.
\end{equation} 
Then the conjugate $(X+Y)^J=J^{-1}(X+Y)J=X^J+Y^J=X-JYJ$ must also commute with $I$.
Using that $Y$ commutes with $K$ we obtain $$X-JYJ=X-JYKI=X-JKYI=X-IYI.$$
The commutation with $I$ is expressed now by $I(X-IYI)=(X-IYI)I$, or
$$IX+YI=XI+IY,$$
which gives $$I(X-Y)=(X-Y)I.$$ 
Adding the last equality to (\ref{fail-of-injectivity}) side by side gives that
$XI=IX$, hence $YI=IY$, which implies $X,Y \in T_eG_\HH$. This
proves the required injectivity of $\widetilde{d\Phi_{(e,e)}}$, which,  
by a slight abuse of notation, we may consider as the inclusion 
$$ T_eG_J/T_eG_{\HH} \oplus T_eG_K/T_eG_{\HH} \subset T_ICompl.$$ 
Now the  surjectivity of $\widetilde{d\Phi_{(e,e)}}$ is equivalent to another inclusion 
$$ T_ICompl \subset T_eG_J/T_eG_{\HH} \oplus T_eG_K/T_eG_{\HH}.$$
Using the natural isomorphism  $T_ICompl \cong 
\{Y\in T_eG\,|\,YI=-IY\}$ in (\ref{Tangent-space}),  we may decompose an arbitrary $Y\in T_ICompl$  into the sum
 of its $J$-commuting and
$J$-anticommuting components, $Y=\frac{1}{2}(Y+Y^J)+\frac{1}{2}(Y-Y^J)$, each of which anticommutes with $I$ and hence determines a tangent vector in $T_ICompl$. Now, since 
$$Y-\frac{1}{2}(Y+Y^J)=\frac{1}{2}(Y-Y^J)$$ anticommutes with both $I$ and $J$, it commutes with $K=IJ$, while $Y+Y^J$ anticommutes with $K$, so that averaging under $K$ both sides of the last equality we get
$\frac{1}{2}(Y-Y^J)=\frac{1}{2}(Y+Y^K)$. Finally
$$Y=\frac{1}{2}(Y+Y^J)+\frac{1}{2}(Y+Y^K) \in  T_eG_J/T_eG_{\HH} \oplus T_eG_K/T_eG_{\HH},$$
which proves the required inclusion and thus the surjectivity of
$\widetilde{d\Phi_{(e,e)}}$.
\end{proof}

\begin{cor}\label{Phisubmers}
Suppose $I,J,K$ is a quaternionic triple. The mapping $\Phi$ is a submersion at $(e,e)\in G_J \times G_K$, that is 
$$d\Phi_{(e,e)}(T_eG_J\oplus T_eG_K) =T_ICompl.$$
\end{cor}
\begin{proof}
The statement that
$\Phi$ is a submersion follows from factoring $d\Phi_{(e,e)}$  
through $\widetilde{d\Phi_{(e,e)}}$
and the fact that the mapping 
$\widetilde{d\Phi_{(e,e)}}: T_eG_J/T_eG_{\HH} \oplus T_eG_K/T_eG_{\HH} \rightarrow T_ICompl$ is an isomorphism by Proposition \ref{Injectivity-proposition}.
\end{proof}

\begin{thm}
\label{Theorem-connectivity}
Given a complex structure $I\in End(V_\RR) $, there is a neighborhood of $I$
in the space of complex structures on $V_\RR$ such that, for any complex structure $I_1$
in this neighborhood, there is a twistor path consisting of three spheres joining $I$ to $I_1$.
Consequently, each connected component of  $Compl$ is twistor path connected. 
\end{thm}
\begin{proof}
Choose a complex structure $J$,  anticommuting with $I$, and consider the sphere $S=S(J,I)$
 and the cone
$C_J$.
By Lemma \ref{lempair}, the complex structures $K=IJ$ and $I$ span the sphere $S=S(K,I)=S(J,I)$. We can then form 
the cone $C_{K}$ whose intersection with $C_J$ contains $S$.
See Picture 1 below where the cones $C_J$ and $C_K$ are depicted by transversal planes
and the sphere $S$ lying in their intersection is depicted by a line.

We first show that 
the images of $C_{K}$ under the action of $G_J$ (``rotation of $C_K$ around $J$'') sweep out an open neighborhood of $I$ in 
$Compl$.
Since
$\Phi$ is a submersion by Corollary \ref{Phisubmers}, there exist neighborhoods  $U_{e,J}\subset G_J$
and $U_{e,K} \subset G_K$ of $e$ such that the set $\Phi(U_{e,J}\times U_{e,K})$
contains an open neighborhood
of $I$.
By definition, the cone $C_K$ contains the orbit $G_K\cdot I$. 
Hence the union $\bigcup_{g\in G_J}{}^gC_K$ contains the image of $\Phi$
and consequently it contains an open neighborhood of $I$.

Now the three twistor spheres connecting $I$ to an arbitrary point $I_1$ in this neighborhood are found as illustrated in the following picture. 

\vspace*{0.8cm}

\input{Picture1-2.pic}

\vspace*{-12cm}
\begin{center}
Picture 1.
\end{center}
\medskip

Finally we conclude that each of the two connected components of  $Compl$ is twistor path connected.
\end{proof}

\subsection{} Another immediate consequence of the injectivity of $\widetilde{d\Phi_{(e,e)}}$ proved in Proposition \ref{Injectivity-proposition}
is the following

\begin{cor}\label{cortrans}
For a quaternionic triple $I,J,K$, the triple intersection of the submanifolds $G_I/G_\HH,G_J/G_\HH$ and $G_K/G_\HH$ of the homogeneous
space $G/G_\HH$ at $eG_\HH$ is transversal.
\end{cor}
The following generalization of this transversality is one of the main ingredients of the proof of connectivity by generic twistor paths in Section 
\ref{Connectivity-by-generic-twistor-paths}.

\begin{prop}
\label{Proposition-general-transversality}
Let $I_1,I_2,I_3$ be complex structures belonging to the same twistor sphere $S$.
The submanifolds $G_{I_1}/G_\HH, G_{I_2}/G_\HH, G_{I_3}/G_\HH$ in $G/G_\HH$
intersect transversally (as a triple) if and only if $I_1,I_2,I_3$ are linearly independent as vectors in 
$End(V_\RR)$.
\end{prop}
\begin{proof} Choose anticommuting 
complex structures $I,J$ in $S$, and set $K=IJ$. 
By Corollary \ref{cortrans},
\begin{equation}
\label{I-J-K-transversality}
T_eG/T_eG_\HH=V_I \oplus V_J \oplus  V_K,
\end{equation}
where we set  $V_I:=T_eG_{I}/T_eG_\HH, V_J:=T_eG_{J}/T_eG_\HH, V_K:=T_eG_{K}/T_eG_\HH$.

We shall prove that $T_eG/T_eG_\HH$ also decomposes into the direct sum of its subspaces 
$V_i:=T_eG_{I_i}/T_eG_\HH, i=1,2,3.$ Put $I_i=a_iI+b_iJ+c_iK, i=1,2,3$. 
Assume, on the contrary, that for certain vectors 
$X \in V_1, Y \in V_2$ and $Z \in V_3$
we have $X+Y+Z=0$. Let $X:=X_I+X_J+X_K$ be the decomposition of $X$ into the sum of its components 
in the respective subspaces of (\ref{I-J-K-transversality}),
and do similarly for $Y$ and $Z$.  Then for $X$ the commutation relation $[X,I_1]=0$
can be written as 
$$a_1[X_J+X_K,I]+b_1[X_I+X_K,J]+c_1[X_I+X_J,K]=0.$$
Note that in the above expression, the term $[X_J,I]$, for example, anticommutes with both $I,J$, hence commutes with $K=IJ$,
and an analogous commutation relation holds for the other terms as well. Hence we can decompose the expression on the left side of the above equality 
with respect to (\ref{I-J-K-transversality}):
$$(b_1[X_K,J]+c_1[X_J,K])+(a_1[X_K,I]+c_1[X_I,K])+(a_1[X_J,I]+b_1[X_I,J])=0.$$
From here we conclude that $b_1[X_K,J]+c_1[X_J,K]=0, a_1[X_K,I]+c_1[X_I,K]=0$
and $a_1[X_J,I]+b_1[X_I,J]=0$. Perturbing
the quaternionic triple $I,J,K$, we may assume that all $a_i, i=1,2,3,$ are nonzero. Then we can use
the last two equalities to express 
\begin{equation}
\label{X-I-reduction}
[X_J,I]=-\frac{b_1}{a_1}[X_I,J], \quad [X_K,I]=-\frac{c_1}{a_1}[X_I,K].
\end{equation}
Note  that $F_J := [\cdot,J] \colon V_I \rightarrow V_K, F_K := [\cdot, K] \colon V_I \rightarrow V_J$
and $F_I := [\cdot, I] \colon V_J \rightarrow V_K$ are isomorphisms of the respective vector spaces. Then, using (\ref{X-I-reduction}),
we can write 
$$X_J=-\frac{b_1}{a_1}F^{-1}_I \circ F_J(X_I), \quad X_K= -\frac{c_1}{a_1}F^{-1}_I \circ F_K(X_I),$$
so that
$$X=X_I +\left(-\frac{b_1}{a_1}F^{-1}_I \circ F_J(X_I)\right) +  \left(-\frac{c_1}{a_1}F^{-1}_I \circ F_K(X_I)\right).$$
Using $a_2,a_3 \neq 0$, we obtain similar expressions for $Y$ and $Z$. 
Since $F_I, F_J, F_K$ are isomorphisms, the equality $X+Y+Z=0$ can now be written as 
\[
\left(\begin{array}{ccc}
1 & 1 & 1\\
-\frac{b_1}{a_1} & -\frac{b_2}{a_2} & -\frac{b_3}{a_3}\\
-\frac{c_1}{a_1} & -\frac{c_2}{a_2} & -\frac{c_3}{a_3}
\end{array}\right)
\left(\begin{array}{c}
X_I \\
Y_I\\
Z_I 
\end{array}\right)=
\left(\begin{array}{c}
0 \\
0\\
0 
\end{array}\right).
\]
This has a nontrivial solution if and only if the columns of the matrix, i.e., $I_1,I_2,I_3$, are linearly dependent.
\end{proof}

\subsection{} We can now prove that the cone $C_I$ has a real analytic structure. Define the incidence correspondence
\[
S_I := \{ (S, J) \mid J\in S \} \subset M_I \times Compl.
\]
Then $S_I$ is an $S^2$-bundle over $M_I$ and $C_I$ is the image of $S_I$ by the projection to $Compl$:
\[
\xymatrix{
N_I \ar@{^{(}->}[r] \ar[rd] & S_I \ar[d]^{pr_1} \ar[r]^-{pr_2} & C_I \subset Compl\\
& M_I.
}
\]
The projection $S_I \ra M_I$ has two sections $\sigma_+$ and $\sigma_-$, given by $+I$ and $-I$ respectively.
\begin{prop}
\label{Prop-cone-dimension}
The real-analytic map  $pr_2: S_I \ra C_I$ is a diffeomorphism away from the images of $\sigma_{\pm}$ and contracts these images to points. Therefore the cone $C_I$ is a real-analytic subset of $Compl$ of dimension $4n^2+1$, smooth away from 
the points $\pm I$.
\end{prop}
\begin{proof}
First note that $pr_2$ clearly contracts the images of $\sigma_{\pm}$. Also, it is injective away from $\pm I$ by Lemma \ref{lempair}. To see that it is also an immersion away from $\pm I$, let $J$ be in $ C_I\setminus\{\pm I\}$, not necessarily anticommuting with $I$.
Define the following mapping
$$\Phi\colon (T_eG_I/T_eG_\HH) \times \RR \rightarrow C_I,$$
$$(X,t) \mapsto e^Xe^{tK}Je^{-tK}e^{-X},$$
where $K \in S(I,J)\setminus S^1$ for
$S^1=\langle I,J\rangle_\RR \cap S(I,J)$.
Then the restriction of $\Phi$ on a small enough neighborhood of $(0,0) \in  (T_eG_I/T_eG_\HH) \times \RR$ defines a parametrization of  $C_I$ around $J$.

Here the subgroup $e^{tK}, t \in \RR$, rotates the sphere $S = S(I,J)$ around the axis $\{\pm K\}$
and, together with the rotation subgroup $e^{tI} \subset G_{I,S} \subset G_I$, sweeps out in $S$, via the above action,
a neighborhood of any point of $S$ other than $\pm I, \pm K$. Proposition \ref{Proposition-general-transversality} provides  that $K$ may be chosen arbitrarily in $S\setminus S^1$, which in turn gives us that $C_I$ is a manifold, smooth away 
from $\pm I$,
of dimension $\dim_\RR (G_I/G_{I,S})+\dim_\RR S=(4n^2-1)+2=4n^2+1$.
The fact that the points $\pm I$ are indeed singular points of the cone $C_I$ is easy to prove.
\end{proof}

\section{Connectivity by generic twistor paths}
\label{Connectivity-by-generic-twistor-paths}

Recall that a period in $Compl$ is {\it generic} if the corresponding complex torus
has trivial N\'eron-Severi group. 
A twistor path in $Compl$
is called {\it generic}, if its successive twistor spheres intersect at generic periods.
In this section we prove the connectivity part of  Theorem \ref{Theorem-tw-path-conn}, i.e.,

\begin{prop} 
\label{Proposition-generic-connectivity}
Any two periods in the  period domain $Compl$ can be connected by a generic twistor path.
\end{prop}

In this section, with the exception of Lemma \ref{Lemma-generic-period} and its proof, we do not assume that the complex structures $I,J,K$ (with or without subscripts) anticommute.

\subsection{Outline of the proof}\label{subsecoutline} 
Define $\mathcal T$ to be the closure, in $Compl \times Compl \times Compl$, of the set of triples $(I,J,K)$ that are linearly independent and belong to the same twistor sphere. Denote by
\[
\begin{array}{cccl}
pr_1\colon & Compl\times Compl \times Compl & \longrightarrow & Compl,\\
pr_{23}\colon & Compl\times Compl \times Compl & \longrightarrow & Compl \times Compl
\end{array}
\]
the respective projections.
For $(I_1,J_1,K_1)\in \cT$, we defined, in Paragraph \ref{subsecPhi}, the mapping $\Phi_{I_1, J_1,K_1}$:
\[
\begin{array}{cccl}
\Phi_{I_1, J_1,K_1} \colon & G_{J_1} \times G_{K_1} & \longrightarrow & Compl, \\
& (g_1,g_2) & \longmapsto & g_1g_2I_1g_2^{-1}g_1^{-1}=\, ^{g_1g_2}I_1.
\end{array}
\]
Proposition \ref{Proposition-general-transversality}
tells us that, when $I_1,J_1,K_1$ are linearly independent, $\Phi_{I_1, J_1,K_1}$ is a submersion near $(e,e) \in G_{J_1} \times G_{K_1}$.
In other words, there is a
neighborhood $U_{e,G} \subset G=GL^+(V_\RR)$ of $e\in G$ 
such that the map $\Phi_{I_1,J_1,K_1}$ is submersive on $U_{e,J_1} \times U_{e,K_1}$, where 
$U_{e,J_1}:=U_{e,G} \cap G_{J_1}$ and $U_{e,K_1}:=U_{e,G} \cap G_{K_1}$
 (and the image is, thus, a neighborhood of $I_1$ in $Compl$).

Let $I_2$ be an arbitrary point in the image of $\Phi_{I_1,J_1,K_1}$ and let $(g_1,g_2) \in U_{e,J_1} \times U_{e,K_1}$ be such that $I_2={}^{g_1g_2}I_1$. 
With this notation, the three twistor spheres connecting $I_1$ to $I_2$ 
are: $S_1:=S(I_1,J_1,K_1)$, $S:={}^{g_1}S_1=S( {}^{g_1}I_1,{}^{g_1}\!J_1=J_1,
{}^{g_1}\!K_1)$ and $S_2:={}^{g_1g_2}S_1=S( I_2,{}^{g_1g_2}J_1,{}^{g_1g_2}K_1 = {}^{g_1}K_1)$,
with the joint points $J_1$ and ${}^{g_1}K$. 

We are going to show 
that,
for a fixed $I_1$, there is a neighborhood $U_{I_1}\subset Compl$ of $I_1$ 
such that for any $I_2\in U_{I_1}$, 
we can choose a generic $J \in C_{I}$, a $K\in S(I,J)$ and find $(g_1,g_2) \in \Phi^{-1}_{I,J,K}(I_1)$ 
as above such that $^{g_1g_2}K$ is also generic.

We begin by proving, in Lemma \ref{Baire-lemma}, that the set of non-generic periods in $C_{I_1}$ is a countable union of proper analytic subsets, i.e., $J$ can be chosen generic.

Next, for $I_2$ close to $I_1$, and with $S_1,S,S_2$ as above, connecting $I_1$ to $I_2$, the initial sphere $S_1$ together with the choice of $J,K\in S_1$, uniquely determines the final sphere $S_2$ together with the pair of periods 
 ${}^{g_1g_2}J,{}^{g_1g_2}K$. 

To justify this uniqueness we first need to control the fibers of the maps $\Phi_{I,J,K}$ in a neighborhood of $(I_1, J_1, K_1)$, which we do in Lemma \ref{Fiber-lemma}.
This allows us to introduce, in Paragraphs \ref{subsecPsi12} and \ref{subsecPsi21}, two maps $\Psi^{I_1\ra I_2}$ and $\Psi^{I_2\ra I_1}$ which, 
roughly speaking, switch $(S_1,J,K)$ and $(S_2,{}^{g_1g_2}J,{}^{g_1g_2}K)$. 

We then show in Lemma \ref{lemPsicompId}, after shrinking our various domains, that the composition of $\Psi^{I_1\ra I_2}$ and $\Psi^{I_2\ra I_1}$ is the identity. Corollary \ref{corfinal} then shows that this implies the irreducibility of the set of triples $(S_1,S,S_2)$ joining $I_1$ and $I_2$ mentioned in the introduction, 
which gives that $J$ and $^{g_1g_2}K$ can both be chosen generic.

Thus the chain of three twistor spheres connecting $I_1$ to $I_2$ for every $I_2$ in some neighborhood of $I_1$
can be chosen in such a way that the periods at the intersections are generic. For arbitrary $I_1$ and $I_2$, we connect them by a path in $Compl$ consisting of generic triple subchains. 

\subsection{} Let us first show that there are generic periods $J \in C_I$. Dimension-wise this is not trivial
because $\dim_\RR C_I=4n^2+1$, whereas the real dimension of the locus of, for example, 
abelian varieties in $Compl$ is $4n^2+2n$. For an alternating form $\Omega$ on $V_\RR$
we denote by $Compl_\Omega$ the locus of periods in $Compl$ at which $\Omega$
represents a class of Hodge (1,1)-type, that is 
$$Compl_\Omega=\{I\in Compl \,|\, \Omega(I\cdot,I\cdot ) = \Omega(\cdot,\cdot)\}.$$ 
If we fix a basis of $V_\RR$ and switch to matrix descriptions, then the condition
$\Omega(I\cdot,I\cdot)=\Omega(\cdot,\cdot)$ simply becomes $^tI\,\Omega I = \Omega$, where $I$ and $\Omega$ also denote the matrices of the corresponding complex structure and alternating form.
The locus of
marked complex tori with nontrivial N\'eron-Severi group is 
$$\mathcal L_{NS}=\bigcup_{0 \neq [\Omega] \in H^2(A,\QQ)} Compl_\Omega,$$
where $A$ is a fixed complex torus.
\begin{lem}
\label{Lemma-open-subset-invariance}
For any alternating form $\Omega$ and any twistor sphere $S$, the intersection $S\cap Compl_\Omega$ is either finite or all of $S$.
\end{lem}

\begin{lem}
\label{Lemma-generic-period}
For any nonzero alternating form $\Omega$, the cone $C_I$ is not contained in $Compl_\Omega$.
\end{lem}

Lemma \ref{Lemma-generic-period} immediately implies
\begin{lem}
\label{Baire-lemma}
For every $I \in Compl$ the set of non-generic periods in $C_I$, that is $C_I \cap \mathcal L_{NS}$, 
is a countable union of closed subsets of $C_I$ none of which contains an open neighborhood (in $C_I$)
of any of its points.
\end{lem}

\begin{proof}[Proof of Lemma \ref{Lemma-open-subset-invariance}]
Follows from the fact that $Compl_\Omega$ and $S$  are complex analytic subsets of $Compl$.

The twistor sphere $S$ is analytic by Corollary \ref{corSanalytic}.
The subset $Compl_\Omega$ is a complex analytic subvariety in $Compl$
as it is the locus where $\Omega$ belongs to the fiber of a
holomorphic subbundle of the Hodge bundle on $Compl$ arising from the Hodge filtration.
\end{proof}

\begin{proof}[Proof of Lemma \ref{Lemma-generic-period}]
We shall prove the following equivalent statement.

For any $J$ anti-commuting with $I$ and any nonzero alternating form $\Omega$ on $V_\RR$ there is a neighborhood $U_\Omega \subset C_{I}$ of 
$J$ such that the locus $Compl_\Omega$ intersects $U_\Omega$
along a real-analytic subvariety of positive codimension.

If $J \notin Compl_\Omega$ there is nothing to prove. Assume $J \in Compl_\Omega$.

From now on we will identify our operators $I,J$ and the form $\Omega$
with their respective matrices corresponding to a choice of basis of $V_\RR$. 
Consider the orbit of $J$ under the conjugation action of  $G_I$: $G_I \cdot J \cong G_I/G_\HH$. 
Let
\[
\begin{array}{cccc}
\Psi \colon & G_I & \longrightarrow & Compl,\\
& g & \longmapsto & ^g\!J=gJg^{-1},
\end{array}
\]
be the evaluation map of the action.
Put $D_{I, J,\Omega} :=\Psi^{-1}(Compl_\Omega)$, that is, $$D_{I, J,\Omega}=\{g \in G_I\,|\, ^t(^g\!J)\Omega 
({}^g\!J) =\Omega\},$$
(note that $D_{I, J,\Omega}$ need not be a subgroup in $G_I$).
Let $g(\tau)$ be any curve in $D_{I, J,\Omega}$ with tangent 
vector $X := g^\prime(0) \in T_eD_{I, J,\Omega}$ at $e = g(0) \in D_{I, J,\Omega}$. Then, differentiating the constant function $^t(^{g(\tau)}\!J)\Omega (^{g(\tau)}\!J)$
at $\tau=0$ we obtain the equality 
\[
\, ^t\! (XJ-JX)\Omega J+\, ^t\! J \Omega (XJ-JX) = 0.
\]
The left hand side may be simplified, given that $^t\! J\Omega J =\Omega$ 
we substitute $\Omega J=({}^t\! J)^{-1}\Omega={}^t\!(J^{-1})\Omega$
into the first summand and ${}^t\! J \Omega=\Omega J^{-1}$ into the second summand, obtaining
$${}^t\! (X^J-X) \Omega + \Omega (X^J-X)=0,$$
where
\[
X^J:=J^{-1}XJ=JXJ^{-1}.
\]
So, denoting $Y:=X^J-X$, we obtain the equality 
\begin{equation}
\label{Equation-locus}
^tY \Omega + \Omega Y=0,
\end{equation}
where $Y$ commutes with $I$ and anticommutes with $J$. 
Note that for any $X\in T_e G_I$, $X = \frac{1}{2} (X + X^J) + \frac{1}{2} (X - X^J)$, where $X+X^J \in T_e G_I$ 
 commutes with $J$ and $X-X^J\in T_e G_I$ anticommutes with $J$. The tangent space $T_e G_\HH$ is the subspace of elements of $T_e G_I$ that commute with $J$. Hence, the subspace of $Y$'s
in $T_eG_I$ anticommuting with $J$ maps isomorphically onto the quotient space $V_I:= T_e{G_I}/T_e{G_\HH}\cong T_{J}\,(G_I\cdot J)$
under the quotient map $T_e{G_I} \rightarrow V_I$. So we need to check
that for a nonzero $\Omega$ the space of solutions to (\ref{Equation-locus}), which is naturally identified with $T_J\,(G_I\cdot J \cap Compl_\Omega)$, has dimension
strictly less than $\dim_\RR T_{J}\,(G_I\cdot J)=\dim_\RR V_I=4n^2$ (i.e., not all of the orbit $G_I\cdot J$ lies in $Compl_\Omega$).

Now conjugate equation (\ref{Equation-locus}) by $I$ to obtain 
$$^tY \Omega^I + \Omega^I Y=0.$$
Adding and subtracting this from (\ref{Equation-locus}) we obtain
$$^tY (\Omega +\Omega^I) + (\Omega+\Omega^I) Y=0\quad \hbox{and} \quad ^tY (\Omega -\Omega^I) + (\Omega-\Omega^I) Y=0.$$
So we may assume that $\Omega$ is either $I$-invariant 
or $I$-anti-invariant in equation (\ref{Equation-locus}). 

\vskip5pt
{\it Case of $I$-invariant $\Omega$.} As $\Omega$ is $J$-invariant, it determines
a skew-symmetric operator $\Omega \colon V_\RR \rightarrow V_\RR$, commuting with $J$. 
So we may choose
an $\Omega$-invariant plane $P=\langle v, Jv\rangle \subset V_\RR$ corresponding to a complex eigenvector 
$v-iJv$ of 
$\Omega \colon V_\RR \rightarrow V_\RR$ such that the matrix of $\Omega|_P$ is
\[\left (\begin{array}{cc}
0& -\lambda \\
\lambda & 0
\end{array}\right).\]
The complex structure $I$ provides another such plane $IP=\langle Iv, JIv \rangle$, which is also $\Omega$-invariant and orthogonal to $P$, so that on $P\oplus IP=\langle v,Jv,Iv,JIv \rangle$ the matrices of
$\Omega, J$ and $I$ are 4$\times$4-block-diagonal with the following blocks on the diagonal
\[\left(\begin{array}{cc|cc} 
0 & -\lambda &   0 & 0\\ 
\lambda & 0 &   0 & 0 \\
\hline
 0 & 0 & 0 & \lambda \\ 
0 & 0 & -\lambda & 0\end{array}\right), 
\left(\begin{array}{cc|cc} 
0 & -1 &   0 & 0\\ 
1 & 0 &   0 & 0 \\
\hline
 0 & 0 & 0 & -1 \\ 
0 & 0 & 1 & 0\end{array}\right),
\left(\begin{array}{cc|cc} 
0 & 0 &   -1 & 0\\ 
0 & 0 &   0 & 1 \\
\hline
 1 & 0 & 0 & 0 \\ 
0 & -1 & 0 & 0\end{array}\right)
.\]
The condition that $Y$ commutes with $I$ and anticommutes with $J$
tells us that $Y$ has a 4$\times$4-block structure with blocks of the form
\[\left(\begin{array}{cc|cc} 
a_1 & a_2 &   b_1 & b_2\\ 
a_2 & -a_1 &   b_2 & -b_1 \\
\hline
 -b_1 & b_2 & a_1 & -a_2 \\ 
b_2 & b_1 & -a_2 & -a_1\end{array}\right)
.\]

Noting that $\Omega = JD=DJ$ for a diagonal matrix $D$ commuting with $J$, we can rewrite  (\ref{Equation-locus}) as 
\begin{equation}
\label{Equation-locus-short-form}
DY=\;  ^tYD, \; Y^I=Y, \; Y^J=-Y.
\end{equation}
For notational convenience we write the matrix $Y$ in terms of its 2$\times$2-blocks $Y_{k,l}$,
$Y=(Y_{k,l})$, $1\leqslant k,l \leqslant 2n$, and denote by $\mathbb{1}_{2\times 2}$ the 2$\times$2 identity
matrix.
If at least one $\lambda_i$, $1\leqslant i\leqslant n$, in $D$ is nonzero we get for all $1\leqslant j \leqslant n$ the equalities of 4$\times$4-blocks
\[
\begin{array}{l}
\left(\begin{array}{cc} 
\lambda_i  {\mathbb 1_{2 \times 2}}& 0\\ 
0 & -\lambda_i {\mathbb 1_{2 \times 2}}    \\
\end{array}\right) \cdot
\left(\begin{array}{cc} 
Y_{2i-1,2j-1} & Y_{2i-1,2j} \\ 
 Y_{2i,2j-1} & Y_{2i,2j}\\
\end{array}\right)=
\\
=
\left(\begin{array}{cc} 
^tY_{2j-1,2i-1} & ^tY_{2j,2i-1} \\ 
 ^tY_{2j-1,2i} & ^tY_{2j,2i}\\
\end{array}\right) \cdot
\left(\begin{array}{cc} 
\lambda_j  {\mathbb 1_{2 \times 2}}& 0\\ 
0 & -\lambda_j {\mathbb 1_{2 \times 2}}    \\
\end{array}\right).
\end{array}
\]
These matrix equalities
completely determine all $n-1$ off-diagonal 4$\times$4-entries of $Y$ in the $i$-th ``fat'' row of 
4$\times$4-blocks
in terms of the off-diagonal 4$\times$4-entries of the $i$-th ``fat'' column, 
$1\leqslant i \leqslant n$. So the codimension of the space of solutions of
(\ref{Equation-locus-short-form}) is at least 
$4(n-1)$ (precise lower bound that is reached in the least restrictive case $\lambda_j=\lambda_i$ for all $j$). For the diagonal 4$\times$4-entry, $i=j$, we obtain $b_2=0$ in $Y_{2i-1,2i}$, so that the codimension
is at least $4n-3$.
\vskip5pt
{\it Case of $I$-anti-invariant $\Omega$:} This is done similarly and leads to the same codimension bound $\geqslant 4n-3$. Alternatively, one could note that, if $C_I \subset Compl_\Omega$, then, in particular, $\pm I \in Compl_\Omega$, so that $\Omega$ is $I$-invariant, and this is the only case we need consider.
\vskip5pt

Now, by Lemma \ref{Lemma-open-subset-invariance}, either a twistor sphere in $C_I$ entirely lies in 
some $Compl_\Omega$ or
its intersection with $\mathcal L_{NS}$ contains only finitely many points of each $Compl_\Omega$.
If $I\not\in \mathcal L_{NS}$ then no twistor sphere in $C_I$ is contained in 
$\mathcal L_{NS}$. The codimension estimate above then allows us to conclude that, for every nonzero $\Omega$, the subset $C_I \cap Compl_\Omega$ is of codimension 
at least $(4n-3)+2=4n-1>0$ in $C_I$. If $I\in \mathcal L_{NS}$, the lower bound for the codimension is still at least 
$4n-3>0$.
The proof is now complete.
\end{proof}


\subsection{} \label{subsecJ12K12}

The transversality of the triple intersection of $G_{I_1}/G_\HH, G_{J_1}/G_\HH, G_{K_1}/G_\HH$
at $eG_\HH$, which is equivalent
to the direct sum decomposition $T_eG/T_eG_\HH=T_eG_{I_1}/T_eG_\HH \oplus T_eG_{J_1}/T_eG_\HH
\oplus T_eG_{K_1}/T_eG_\HH$, is preserved if we perturb $(I_1,J_1,K_1)\in \cT$ a little. In other words, there is a compact neighborhood
$U_{I_1,J_1,K_1}\subset \mathcal T$ of $(I_1,J_1,K_1)$ and a compact neighborhood
$U_{e,G}\subset G$ such that $\Phi_{I,J,K}\colon U_{e,J} \times U_{e,K} \rightarrow Compl$ 
is a submersion onto its image for all $(I,J,K)\in U_{I_1,J_1,K_1}$. Moreover,
 there is a compact neighborhood $U_{I_1}\subset Compl$ of $I_1$ which is contained in the image
$\Phi_{I,J,K}(U_{e,J}\times U_{e,K})$ for all $(I,J,K)\in U_{I_1,J_1,K_1}$. We will always assume
that for each neighborhood $U_{e,G}$ we made a choice of such $U_{I_1}=U_{I_1}(U_{e,G})$. 
Note that every $I_2\in U_{I_1}$ is a regular value of $\Phi_{I,J,K}$ for all 
$(I,J,K)\in  U_{I_1,J_1,K_1}$.

\begin{lem}
\label{Fiber-lemma}
There exists a  neighborhood $U_{e,G}$ such that
for all $I_2 \in U_{I_1}$ and for all $(I,J,K)\in U_{I_1,J_1,K_1}$, 
the full preimage $\Phi_{I,J,K}^{-1}(I_2)$ 
is an $8n^2$-dimensional submanifold in $U_{e,J} \times U_{e,K}$
of the form 
\begin{equation}
\label{Fiber-equation}
\{(f_1h_1, h_1^{-1}f_2h_2)\mid h_1,h_2 \in G_\HH\} \cap (U_{e,J} \times U_{e,K}),
\end{equation} 
where $(f_1,f_2)$ is a pair in $U_{e,J} \times U_{e,K}$ such that $\Phi_{I,J,K}(f_1,f_2)=I_2$.
\end{lem}

\begin{proof}[Proof of Lemma \ref{Fiber-lemma}] The fact that 
$\Phi_{I,J,K}^{-1}(I_2) \cap(U_{e,J} \times U_{e,K})$ consists of 
a finite number of $8n^2$-dimensional manifolds follows from the regularity of $I_2$.

While the part of the fiber in (\ref{Fiber-equation}) may have been easily guessed, the fact that for a small enough $U_{e,G}$ this is the whole fiber 
follows from Proposition \ref{Proposition-general-transversality}.
Indeed, assuming that we have $(f_1,f_2), (g_1,g_2) \in 
\Phi^{-1}_{I,J,K}(I_2)\subset G_{J}\times G_{K}$, we see that $f_2^{-1}f_1^{-1}g_1g_2 \in G_{I}$.
Setting $g_{I}=f_2^{-1}f_1^{-1}g_1g_2$ and $g_{J}=f_1^{-1}g_1 \in G_{J}$ we 
have the equality
$$f_2g_{I}=g_{J}g_2.$$
The left side of the equality lies in $G_KG_{I}$ and the right side lies in $G_JG_K$.
If we restrict
ourselves to $\Phi^{-1}_{I,J,K}(I_2) \cap (U_{e,G}\times U_{e,G})$ for a small enough neighborhood $U_{e,G} \subset G$ then Proposition  \ref{Proposition-general-transversality} tells us that, 
for every element  in the product $U_{e,J}U_{e,K}U_{e,I}$,
each of its three factors is uniquely determined up to a $G_\HH$-correction. 

So from our equality $f_2g_{I}=g_{J}g_2$ we obtain
$g_{I}, g_{J} \in G_\HH$, which, after setting 
$h_1 := g_J=f_1^{-1}g_1$ and $h_2 :=g_I$, implies
that $g_1=f_1h_1$ and $g_2=g_J^{-1}f_2g_I=h_1^{-1}f_2h_2$.

Since $U_{I_1},U_{I_1,J_1,K_1},U_{e,G}$ are compact and $U_{e,G}$ is independent of the choice of $(I,J,K)\in U_{I_1,J_1,K_1}$,
there is a universal upper bound for the number of connected components
of $\Phi_{I,J,K}^{-1}(I_2)\cap (U_{e,J} \times U_{e,K})$, for all 
$I_2\in U_{I_1}$ and all $(I,J,K)\in U_{I_1,J_1,K_1}$. Therefore we can shrink the compact neighborhood
$U_{e,G}$ so that the fibers $\Phi_{I,J,K}^{-1}(I_2) \cap (U_{e,J} \times U_{e,K})$ for all 
$I_2\in U_{I_1}$ and all $(I,J,K)\in U_{I_1,J_1,K_1}$ contain only the component specified in (\ref{Fiber-equation}).
\end{proof}
Regarding the proof of Lemma \ref{Fiber-lemma}, we note the following.
\begin{rem}
\label{Fiber-deformation-remark}
It is not hard to see that the  fiber $\Phi_{I,J,K}^{-1}(I_2)$ in Lemma \ref{Fiber-lemma},
as a topological subspace of $G \times G$, depends continuously on 
$I_2\in U_{I_1}$ and  $(I,J,K)\in U_{I_1,J_1,K_1}$.
\end{rem}

\begin{rem}
In general,  it is possible that $g \in U_{e,G}$ is not uniquely representable as a triple product 
of elements  in the larger sets $G_{J}$, $G_{K}$, $G_{I}$ 
and thus we cannot say if the whole fiber $\Phi^{-1}_{I,J,K}(I_2) \subset G_{J}\times G_{K}$ consists of just one
$G_\HH \times G_\HH$-orbit as in Lemma \ref{Fiber-lemma}. This is why we possibly need to shrink $U_{e,G}$.
\end{rem}

\subsection{}\label{subsecPsi12} Recall that, for any $I$, $M_I=G_{I}/G_{I,S}$ parametrizes the twistor lines through $I$ (see Paragraph \ref{subsecSMN}). For all $I$, put $U_{I_1,J_1,K_1}(I) :=pr_{23}(pr_1^{-1}(I)\cap U_{I_1,J_1,K_1})$.
Then $U_{I_1,J_1,K_1} (I_1)$ is a neighborhood of
$(J_1,K_1)$ in $C_{I_1}\times_{M_{I_1}} C_{I_1} = pr_{23} (pr_1^{-1} (I_1) \cap \cT)$.
Consider the map 
\[
\begin{array}{cccl}
\Psi^{I_1 \rightarrow I_2} \colon & U_{I_1,J_1,K_1} (I_1) &
\longrightarrow & {C_{I_2}}
\times_{M_{I_2}} C_{I_2} = pr_{23} (pr_1^{-1} (I_2) \cap \cT),\\
& (S(J,K),J,K) & \longmapsto &(S({}^{f_1f_2}J, {}^{f_1f_2}K), {}^{f_1f_2}K, {}^{f_1f_2}J),
\end{array}
\]
where $ (f_1,f_2) \in \Phi^{-1}_{I_1,J,K}(I_2) \cap ({U_{e,J} \times U_{e,K}})$,
and we use, in an obvious way, the triple notation of the kind $(S(J,K),J,K)$ for the elements of the 
fiber products above. Note the switched order of ${}^{f_1f_2}K, {}^{f_1f_2}J$.
The role of this change of order will be clarified later.

Lemma \ref{Fiber-lemma} guarantees that the mapping $\Psi^{I_1 \rightarrow I_2}$
is well-defined, as its value at $(S(J,K),J,K)$ is uniquely determined by the fiber 
$\Phi^{-1}_{I_1,J,K}(I_2) \cap ({U_{e,J} \times U_{e,K}})$, 
so it does not depend on the choice of a particular point in the fiber.

\vspace*{1cm}
\hspace*{40mm}
\input{Picture4-2.pic}

\vspace*{-11.5cm}
\begin{center}
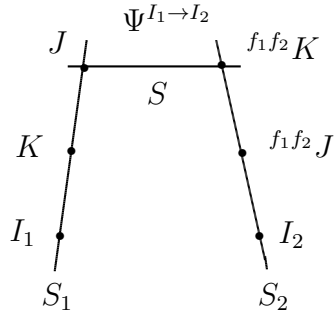
Picture 2: For fixed $I_1$ and $I_2$ any pair $(J,K) \in C_{I_1}\times_{M_{I_1}} C_{I_1}$ near $(J_1,K_1)$ determines a unique pair 
$({}^{f_1f_2}K,\,^{f_1f_2}\!J)\in C_{I_2}\times_{M_{I_2}} C_{I_2}$.
\end{center}
\medskip
\subsection{} 
\label{Inverse-preparation-subsection}
Next, for each $(I,J,K) \in U_{I_1,J_1,K_1}$  consider the mapping 
$\Phi_{I,K,J}$ (note that  we switched $J$ and $K$ in the subscript).
By shrinking the original $ U_{I_1,J_1,K_1}$ and $U_{e,G}$ if needed,
we can find a compact neighborhood $V_{e,G} \subset G$ such that
 for each $(I,J,K) \in U_{I_1,J_1,K_1}$ we have

(a) $\Phi_{I,K,J}\colon V_{e,K} \times V_{e,J} \rightarrow Compl$ is a submersion
onto its image;

(b) every fiber of this mapping is of the form described in Lemma \ref{Fiber-lemma};

and 

(c) the image $\Phi_{I,K,J}( V_{e,K} \times V_{e,J})$ contains 
$U_{I_1} \subset  \underset{(I,J,K)\in U_{I_1,J_1,K_1}}{\bigcap}\Phi_{I,J,K}(U_{e,J}\times U_{e,K})$ (see Paragraph \ref{subsecJ12K12}).
\vskip5pt
By Lemma \ref{Fiber-lemma}, conditions (a) and (b) are satisfied. We need only to comment on (c). By Lemma \ref{Fiber-lemma}, for the original triple $(I_1,J_1,K_1) \in U_{I_1,J_1,K_1}$, we can find $V_{e,G}$ such that 
$\Phi_{I_1,K_1,J_1}\colon  V_{e,K_1} \times V_{e,J_1} \rightarrow Compl$, where
$V_{e,K_1}:=V_{e,G} \cap G_{K_1},V_{e,J_1}:=V_{e,G} \cap G_{J_1}$, satisfies (a) and (b). Shrinking $U_{e,G}$
and, thus, $U_{I_1}$, if needed,
we can satisfy (c) for $\Phi_{I_1,K_1,J_1}$. Now shrinking $U_{I_1,J_1,K_1}$ and again $U_{e,G}$, if needed, we can satisfy conditions (a), (b) and (c)
for all $(I,J,K)\in U_{I_1,J_1,K_1}$.


\subsection{} 
\label{Domain-preparation-subsection}
Now introduce $V_{I_1,K_1,J_1}:=\{(I,K,J)\,|\, (I,J,K)\in U_{I_1,J_1,K_1}\}$
and $V_{I_1,K_1,J_1}(I):=pr_{23}(pr_1^{-1}(I) \cap V_{I_1,K_1,J_1})$.

Then, for all $(I,K,J)$ in the interior of $V_{I_1,K_1,J_1}$, the set
$pr_1(V_{I_1,K_1,J_1})$ is a neighborhood of $I$ in $Compl$ and $V_{I_1,K_1,J_1}(I)$ is a neighborhood of $(K,J)\in C_I\times_{M_I} C_I$.
Note that, due to Condition (c) in Paragraph \ref{Inverse-preparation-subsection},
for all $I \in U_{I_1}\cap pr_1(V_{I_1,K_1,J_1})$ and for all $(K,J)\in V_{I_1,K_1,J_1}(I)$,
the image $\Phi_{I,K,J}(U_{e,K}\times U_{e,J})$ contains the neighborhood $U_{I_1}$.

\subsection{}\label{subsecPsi21} Choose $I_2 \in U_{I_1}\cap pr_1(V_{I_1,K_1,J_1})$ and $K,J$ such that $(I_2,K,J) \in V_{I_1,K_1,J_1}$. Conditions (a),(b) and (c) in Paragraph \ref{Inverse-preparation-subsection} allow us to define, analogously to $\Psi^{I_1 \rightarrow I_2}$, the map
\[
\begin{array}{cccc}
\Psi^{I_2 \rightarrow I_1} \colon &
V_{I_1,K_1,J_1}(I_2) & \longrightarrow & {C_{I_1}}\times_{M_{I_1}} C_{I_1},\\
& (S(J,K),K,J) & \longmapsto & (S({}^{d_1d_2}J, {}^{d_1d_2}K),{}^{d_1d_2}J,{}^{d_1d_2}K),
\end{array}
\]
for $(d_1,d_2) \in \Phi^{-1}_{I_2,K,J}(I_1)$ (again, note the reversed order of $J$ and $K$ in the subscript). 

The period $^{f_1f_2}K$ in Picture 2 above will play the role of the ``rotation center'' 
for  $\Phi_{I_2,\,{}^{f_1f_2}K,\,{}^{f_1f_2}J}$ (here $J,K \in C_{I_1}$),
similar to the role that $J$ plays  for $\Phi_{I_1,J,K}$.
This explains why we switched $J$ and $K$.

Below we will impose restrictions on the domain of $\Psi^{I_1 \rightarrow I_2}$ 
in order for the image of this map to be contained in the domain of $\Psi^{I_2 \rightarrow I_1}$, 
so that we can compose them.

We begin by choosing a compact neighborhood
$U_{J_1,K_1}$ of $(J_1,K_1)$ in $U_{I_1,J_1,K_1} (I_1)$, which can at first be all of $U_{I_1,J_1,K_1} (I_1)$. 
We will later modify $U_{J_1,K_1}$, without changing the original $U_{I_1,J_1,K_1}$.

\begin{lem}\label{lemPsidomains}
For fixed $V_{I_1,K_1,J_1}$, we can shrink $U_{e,G}$ and 
 $U_{J_1,K_1}$ so that for arbitrary $I_2 \in U_{I_1}$,
\[
\Psi^{I_1 \rightarrow I_2}(U_{J_1,K_1}) \subset V_{I_1,K_1,J_1}(I_2).
\]
\end{lem}
 
\begin{proof}
As in Paragraph \ref{Inverse-preparation-subsection}, this
follows from the fact that the mapping $\Psi^{I_2\rightarrow I_1}$
depends continuously on  $I_2$ (see Remark \ref{Fiber-deformation-remark}), and that 
 $$\underset{I_2\to I_1}{\lim}\Psi^{I_2\rightarrow I_1}=(12)\colon U_{J_1,K_1} \rightarrow 
V_{I_1,K_1,J_1}(I_1),$$ $$(S(J,K),J,K)\mapsto (S(J,K),K,J),$$ the latter mapping is trivially defined on the whole 
$U_{J_1,K_1}$,
so that the sizes of the domains $V_{I_1,K_1,J_1}(I_2)$ of $\Psi^{I_2\rightarrow I_1}$'s are bounded away from zero, when $I_2$ is close to $I_1$. 

As before, we can further shrink $U_{e,G}$ (and hence $U_{I_1}$), if needed, so that
properties (a), (b), (c) in Paragraph \ref{Inverse-preparation-subsection} hold independently of the point $I_2\in U_{I_1}$.
\end{proof}

\subsection{} Possibly shrinking $U_{e,G}$, we can and will assume that it is invariant under taking inverses,
$g\mapsto g^{-1}$. 

\begin{lem}\label{lemPsicompId}
Possibly further shrinking $U_{e,G}$  and $U_{J_1,K_1}$, satisfying the conclusion of Lemma \ref{lemPsidomains}, we have for all $I_2\in U_{I_1}$
\[
\Psi^{I_2\rightarrow I_1} \circ \Psi^{I_1 \rightarrow I_2} = 
Id|_{U_{J_1,K_1}}.
\]
\label{lemma-composition}
\end{lem}
\begin{proof}
For all $(f_1,f_2) \in \Phi^{-1}_{I_1,J,K}(I_2)\cap (U_{e,G} \times U_{e,G})$ and all $(S(J,K),J,K) \in 
U_{J_1,K_1} 
$, we want the neighborhoods
$V_{e, \,{}^{f_1f_2}K}=V_{e,G} \cap G_{{}^{f_1f_2}K}, 
V_{e,\,{}^{f_1f_2}J}=V_{e,G} \cap G_{{}^{f_1f_2}J}$ to contain, respectively, the neighborhoods
${}^{f_1f_2}U_{e,K}=f_1f_2U_{e,K}f_2^{-1}f_1^{-1}$ and ${}^{f_1f_2}U_{e,J}=f_1f_2U_{e,J}f_2^{-1}f_1^{-1}$,
so that, in particular, $V_{e,\, {}^{f_1f_2}K} \times V_{e,\,{}^{f_1f_2}J}$
contains the pair
\[
(d_1,d_2)=(f_1f_2\cdot f_2^{-1} \cdot f_2^{-1}f_1^{-1}, f_1f_2 \cdot f_1^{-1} \cdot f_2^{-1}f_1^{-1})= (f_1f_2^{-1}f_1^{-1}, f_1f_2 \cdot f_1^{-1} \cdot f_2^{-1}f_1^{-1}).\]
Here the invariance of $U_{e,G}$ under taking inverses is used. The pair $(d_1, d_2)$
certainly belongs to the preimage $\Phi^{-1}_{I_2,\, {}^{f_1f_2}K,\,{}^{f_1f_2}J}(I_1)$ as the product
of its entries is $f_1f_2^{-1}f_1^{-1} \cdot f_1f_2 \cdot f_1^{-1} \cdot f_2^{-1}f_1^{-1}=f_2^{-1}f_1^{-1}$.

Note that, 
for $U_{e,G}$ small enough, the neighborhoods ${}^{f_1f_2}U_{e,K} \times {}^{f_1f_2}U_{e,J}$
will also be uniformly small for all 
$(f_1,f_2) \in \Phi_{I_1,J,K}^{-1}(I_2) \cap (U_{e,G} \times U_{e,G})$, so that
the fiber of 
$\Phi_{I_2,\,{}^{f_1f_2}K,\,{}^{f_1f_2}J}^{-1}(I_1)$ in ${}^{f_1f_2}U_{e,K} \times {}^{f_1f_2}U_{e,J}$ 
consists of a unique connected component
of the form described in Lemma \ref{Fiber-lemma}. Then the pair $(d_1,d_2)$
is contained in this ``good'' part of the fiber $\Phi_{I_2,\,{}^{f_1f_2}K,\,{}^{f_1f_2}J}^{-1}(I_1)$ and
we can use $(d_1,d_2)$
to evaluate $\Psi^{I_2 \mapsto I_1}$ at $(S({}^{f_1f_2}J,{}^{f_1f_2}K), 
{}^{f_1f_2}J, {}^{f_1f_2}K)$.
Thus $$\Phi_{I_2,\,{}^{f_1f_2}K,\,{}^{f_1f_2}J}(d_1,d_2)=I_1$$ and $${}^{d_1d_2f_1f_2}J=J, 
{}^{d_1d_2f_1f_2}K=K,$$
so that $$\Psi^{I_2 \mapsto I_1}(S({}^{f_1f_2}J,{}^{f_1f_2}K),{}^{f_1f_2}K, {}^{f_1f_2}J)=(S(J,K),J,K),$$
where, certainly, $S(J,K)=S(I_1,J,K)$,
proving that the composition $\Psi^{I_2\rightarrow I_1} \circ \Psi^{I_1 \rightarrow I_2}$ is the identity
on $U_{J_1,K_1}$.

In order to ensure that $V_{e,\, {}^{f_1f_2}K} \times V_{e,\,{}^{f_1f_2}J}$ contains $(d_1,d_2)$,
we assume, shrinking $U_{e,G}$ and $U_{J_1,K_1}$ 
if necessary, but not changing $V_{e,G}$ and the previously fixed $V_{I_1,K_1,J_1}$, that
for all $(S(J,K),J,K) \in U_{J_1,K_1}
$ and for all points $(f_1,f_2) \in 
\Phi^{-1}_{I_1,J,K}(I_2)\cap (U_{e,G} \times U_{e,G})$, the neighborhoods
$V_{e,\, {}^{f_1f_2}K}, V_{e,\,{}^{f_1f_2}J}$ contain, respectively, the neighborhoods
${}^{f_1f_2}U_{e,K}$ and ${}^{f_1f_2}U_{e,J}$.
\end{proof}

\begin{cor}\label{corfinal} Let $U_{I_1}$ be defined by $U_{e,G}$ ($U_{e,G}$ satisfying Lemma \ref{lemma-composition}).
For arbitrary $I_2 \in U_{I_1}$, both joint points $J \in C_{I_1}$ and  ${}^{f_1f_2}K\in C_{I_2}$ of a triple of twistor spheres connecting $I_1$ and $I_2$, can be chosen generic.
\end{cor}
\begin{proof}
Define
\[
\begin{array}{cccc}
pr_K \colon & V_{I_1,K_1,J_1}(I_2)
& \longrightarrow & C_{I_2} \subset Compl, \\
& (S(J,K),K,J)&\longmapsto &K.
\end{array}
\]
This projection is a submersion onto its image. 
By Lemma  \ref{Baire-lemma}, the locus $\mathcal L_{NS}$ intersects $C_{I_2}$
in a countable union of closed submanifolds of positive codimension in $C_{I_2}$. 
As the mapping $pr_K$
is a submersion onto its image, the preimage $pr_K^{-1}(\mathcal L_{NS}\cap C_{I_2})$ is also a countable union of 
closed submanifolds of positive codimension in $V_{I_1,K_1,J_1}(I_2)$. 
Similarly, for 
\[
\begin{array}{cccc}
pr_J \colon & 
U_{J_1,K_1} & \longrightarrow & C_{I_1}\subset Compl, \\
& (S(J,K),J,K) & \longmapsto & J,
\end{array}
\]
$pr_J^{-1}(\mathcal L_{NS}\cap C_{I_1}) \subset 
U_{J_1,K_1}$
is a countable union of closed submanifolds of positive codimension. 
The mapping $\Psi^{I_2\to I_1}$ is real-analytic, so the closure of
$\Psi^{I_2\to I_1} (pr_J^{-1}(\mathcal L_{NS}\cap C_{I_2}))$ in $U_{J_1,K_1}$
does not contain interior points. 
Therefore
\begin{equation}
\label{inequality-irreducibility}
U_{J_1,K_1}
\neq pr_J^{-1}(\mathcal L_{NS}\cap C_{I_1}) \cup \Psi^{I_2 \rightarrow I_1}(pr_K^{-1}
(\mathcal L_{NS}\cap C_{I_2})).
\end{equation}
Since, by Lemma \ref{lemma-composition}, $\Psi^{I_2\to I_1}\circ \Psi^{I_1\to I_2}=Id|_{U_{J_1,K_1}}$,
the inequality (\ref{inequality-irreducibility}) tells us that
the image of the mapping $\Psi^{I_1\to I_2}$ is not contained in $pr_K^{-1}(\mathcal L_{NS}\cap C_{I_2})$. 
Thus we may find
a pair $(J,K)\in 
U_{J_1,K_1}$ such that $J=pr_J(S(J,K),J,K)\notin \mathcal L_{NS}\cap C_{I_1}$ and 
${}^{f_1f_2}K=pr_K(\Psi^{I_1 \rightarrow I_2}(S(J,K),J,K)) \notin \mathcal L_{NS}\cap C_{I_2}$, 
that is, both periods are generic.
\end{proof}

\section{The degree of twistor lines}
\label{Twistor-spheres-are-complex-submanifolds}

In this section we show that twistor lines in $Compl$ have degree $2n$ in the Pl\"ucker embedding. We first show that the group $G=GL(V_\RR)$ acts transitively on the set of twistor lines in $Compl$ and then compute the degree of an explicit twistor line.

\begin{lem}\label{lemtrans}
The group $G=GL(V_\RR)$ acts transitively on the set of twistor lines in $Compl$.
\end{lem}

\begin{proof}
Given two twistor spheres $S_1 = S(I_1, J_1)$ and $S_2= S(I_2, J_2)$, there is an element $g\in G$ sending $I_1$ to $I_2$, hence sending $S_1$ to a twistor sphere through $I_2$. The lemma now follows from Corollary \ref{corMItrans}.
\end{proof}

\subsection{} To construct our example, consider the affine chart in the Grassmannian $Gr(2n,4n)$
of normalized period matrices $(\mathbb{1}| Z)$, where $\mathbb{1}$ is, in general, the $2n\times 2n$
identity matrix and $Z$ now denotes a non-degenerate $2n\times 2n$ complex matrix.
Let us fix a basis of $V_\RR$ and write the matrix of an arbitrary complex structure 
$I\colon V_\RR \rightarrow V_\RR$ in the following block form
\[I=\left(\begin{array}{cc}
A & B \\
C & D
\end{array}\right),\]
for $2n\times2n$ real matrices $A,B,C,D$.
Then the relation 
$$(\mathbb{1}\,|\,Z)I=(i\mathbb{1}\,|\,i Z),$$
gives the matrix equations 
\[ A+Z C=i\mathbb{1}, \quad B+Z D=i Z.
\]
Assume that $C$ is invertible so that the first equation
allows us to write $Z =(i\mathbb{1}-A)C^{-1}$.
The condition that $I$ is a complex structure will then guarantee that the second
equation is automatically satisfied.
\medskip

\subsection{The case $n=1$}\label{subsecn=1} Momentarily assume $n=1$ and consider the twistor sphere $S = S(I,J)$ where $I$ and $J$ have the respective matrices
\[
\left(\begin{array}{cc|cc} 
0 & -1 &   0 & 0\\ 
1 & 0 &   0 & 0 \\
\hline
 0 & 0 & 0 & -1 \\ 
0 & 0 & 1 & 0\end{array}\right),
\left(\begin{array}{cc|cc} 
0 & 0 &   -1 & 0\\ 
0 & 0 &   0 & 1 \\
\hline
 1 & 0 & 0 & 0 \\ 
0 & -1 & 0 & 0\end{array}\right)
\]
and put $K = IJ$.
So for $\lambda \in S$,
\[
\lambda =aI+bJ+cK=
\left(\begin{array}{cc|cc} 
0 & -a &   -b & -c\\ 
a & 0 &   -c & b \\
\hline
 b & c & 0 & -a \\ 
c & -b & a & 0
\end{array}\right).
\]
Assume additionally that $b^2+c^2 \neq 0$, that is, $\lambda \in S\setminus\{\pm I\}$.
Here 
\[
A=
\left(\begin{array}{cc}
 0 & -a\\ 
 a & 0 \\
\end{array}\right),
C=
\left(\begin{array}{cc}
 b & c\\ 
 c & -b \\
\end{array}\right),
C^{-1}=
\frac{1}{b^2+c^2}
\left(\begin{array}{cc}
 b & c\\ 
 c & -b \\
\end{array}\right).
\]
Then 
\[ Z = (i\mathbb{1}-A)C^{-1}=
\frac{1}{b^2+c^2}
\left(\begin{array}{cc}
ac+ib & -ab+ic\\ 
 -ab+ic & -ac-ib \\
\end{array}\right)=
\left(\begin{array}{cc}
 z_1 & z_2\\ 
 z_3 & z_4 \\
\end{array}\right),
\]
where $Z$ clearly satisfies the equations $$z_1+z_4=0,\,\,z_2-z_3=0,\,\,\det \,Z=z_1z_4-z_2z_3=1.$$

\subsection{} Now, for a general $n$, we can 
construct a twistor line in the period domain of complex $2n$-dimensional tori,
which, in the affine chart of $Gr(2n,4n)$ above, corresponds to the
locus of matrices $(\mathbb{1}| Z)$ where $Z$ is the block-diagonal matrix with the same $2\times 4$-block
\[\left(\begin{array}{cc|cc}
 1 & 0 &  u & v\\ 
 0 & 1 & v & -u \\
\end{array}\right), u^2+v^2=-1,
\]
on the diagonal.

\subsection{}\label{pardegree}
The degree of the curve in the example is $2n$ under the Pl\"ucker embedding 
$Gr(2n,4n) \hookrightarrow \mathbb{P}^{{4n \choose{2n}}-1}$. 
Indeed, the Pl\"ucker coordinates in the above affine chart are given by the maximal minors of the matrix $(\mathbb{1}| Z)$. The twistor line $S$ in our example is contained in the plane $P(S)$ with parameters $u,v$ in the part given by the affine chart. 
Let $W$ be the homogeneous coordinate given by the minor formed by all $\mathbb{1}_{2\times2}$-blocks and let $U$ and $V$
be any homogeneous coordinates such that after restricting to $P(S)$ we get $u=\frac{U}{W}$ and $v=\frac{V}{W}$.

Let us consider from now on the plane $P(S)$ as a projective (complete) 2-plane in $Gr(2n,4n)$ with coordinates $U,V,W$. 
Then the minor formed by the $(u,v)$-blocks restricts to $P(S)$ as $(U^2+V^2)^n$.  
Rewriting the equation $u^2+v^2=-1$ of our twistor line $S$ in homogeneous coordinates we get $U^2+V^2+W^2=0$,
so that, restricting the polynomial $(U^2+V^2)^n$ to $S$, we see that it vanishes
precisely when $U^2+V^2=-W^2$ vanishes, that is, only at the points $\pm I$ of $S$ outside of our affine chart.
Each of the two factors in the expansion $(U^2+V^2)^n=(U+iV)^n(U-iV)^n$ vanishes to order $n$ at the respective point,
so the total order of vanishing is $n+n=2n$ which is the degree of the image of $S$ under the Pl\"ucker embedding.

\begin{cor}
\label{CorDegree}
Twistor lines have degree $2n$ in the Pl\"ucker embedding of $Compl$.
\end{cor}
\begin{proof}
Follows from Lemma \ref{lemtrans} and Paragraph \ref{pardegree}.
\end{proof}

\begin{rem}\label{remreferee}
There is an alternative proof of the above corollary following the lines explained in Remark \ref{RemPlucker}.
Namely, we have the embedding $i \colon Gr(2,\mathbb{H}\otimes \CC) \hookrightarrow Gr(2n,V_\CC),
\mathbb{H}^{1,0}\mapsto \mathbb{H}^{1,0}\otimes V^{\prime}_\CC$. Let $e_1,\dots,e_n$
be some basis of $V^{\prime}_\CC$. Then, in terms of the Pl\"ucker embeddings of the respective Grassmanians,
we have $i(u\wedge v)=u\otimes e_1 \wedge v \otimes e_1 \wedge \dots \wedge u \otimes e_n\wedge v \otimes e_n$,
which induces an isomorphism $i^{\ast}(\mathcal O_{Gr(2n,V_\CC)}(1))\cong 
\mathcal O_{Gr(2,\mathbb{H}\otimes \CC)}(n)$ of the sheaves on the quadric $Gr(2,\mathbb{H}\otimes \CC)$, thus justifying that the degree of 
$S\subset Gr(2n,4n)$ under the Pl\"ucker embedding is $2n$. 
\end{rem}

\section{Twistor path connectivity of $Compl_\Omega$}
\label{Twistor-path-connectivity-of-Compl-Omega}

\sloppy
In this section we  describe the discrete invariants and the homogeneous structure of the connected components of
$Compl_\Omega$ and  
establish a criterion of twistor path connectivity of a 
``special'' connected component of $Compl_\Omega$ 
(the remaining ``non-special'' components will be shown  to contain no  twistor lines at all). 

Let $I$ be a complex structure operator in $Compl_\Omega$,
 that is, $\Omega(Iu,Iv)=\Omega(u,v)$ for all $u,v\in V_\RR$. 
On the vector space $(V_\RR,I)$, considered as a complex vector space, the form $\Omega$ 
determines a hermitian form $$h(u,v):=\Omega(u,Iv)-i\Omega(u,v)$$ (note that $h(u,Iv)=-ih(u,v),h(Iu,v)=ih(u,v)$), which we will call
{\it the hermitian form associated to $\Omega$ and $I$}.  

The signature of $h$ is a triple $(n_+,n_-,n_0)$, where $n_+$,$n_-$
and $n_0$ are the complex dimensions of, respectively, a maximal positive subspace $V_+$, a maximal negative subspace $V_-$, and the null subspace 
$$V_0=\{u\in V_\RR\,|\,h(u,v)=0\mbox{ for all } v\in V_\RR\}
=\{u\in V_\RR\,|\,\Omega(u,v)=0\mbox{ for all } v\in V_\RR\}$$ 
of $h$ in $(V_\RR,I)$, so that $n_++n_-+n_0=2n$.
The subspaces $V_+,V_-$ and the numbers $n_+$ and $n_-$ depend, in general, on 
the choice of $I \in Compl_\Omega$, while 
$V_0$, and hence $n_0$, depend only on $\Omega$, $V_0$ being invariant with respect to every complex structure operator in $Compl_\Omega$. 

The group $G = GL(V_\RR)$  acts naturally on 2-forms on 
$V_\RR\times V_\RR$, $$g\colon f(u,v) \mapsto 
g^*f(u,v)=f(gu,gv)$$ or, in matrix notation,
$g\colon f \mapsto {}^t\!g\cdot f \cdot g$.  

Consider the  subgroup 
$G_\Omega \subset G$ of automorphisms of $\Omega$: $$G_\Omega :=\{g \in G\,|\, g^*\Omega= \Omega\}\subset G.$$
The subgroup $G_\Omega$ acts via the $G$-action on $Compl_\Omega$.           For $I \in Compl_\Omega$ we denote by $G_{I,\Omega}$ the stabilizer of $I$ in $G_\Omega$:
\[
G_{I,\Omega}=\{g\in G_\Omega\,|\,{}^g\!I=I\}.
\]
We note that, by definition, for the hermitian form $h$ associated to 
$\Omega$ and $I$, the subgroup $G_{I,\Omega}$ is the 
stabilizer of $h$ in $G_I$ under the corresponding $G$-action, i.e., $G_{I,\Omega}=\{g\in G_I\,|\,g^*h=h\}$. 

Note that if $\Omega$ is non-degenerate, then $G_\Omega$ is connected,
and  if $\Omega$ is degenerate, that is, $n_0>0$, then the group  $G_\Omega$  consists of two connected components, one in each connected component of $G$.
We denote by $G_\Omega^0$  the connected subgroup component of $G_\Omega$. 

We set $Compl_\Omega^{\pm} :=Compl_\Omega \cap Compl^{\pm}$. 
Consider the finite set
\[
S_\Omega :=\{(k,l,n_0)\,|\,k,l\in \mathbb{Z}_{\geqslant 0}, k+l+n_0=2n\}
\]
consisting of $2n-n_0+1$ triples,  and the mapping
\[
\begin{array}{rcl}
Sign\colon Compl_\Omega^{\pm} & \longrightarrow & S_\Omega, \\
I & \longmapsto & \mbox{ the signature }(n_+,n_-,n_0) \mbox{ of }h
 \mbox{ assoc. to } \Omega \mbox { and } I.
 \end{array}
 \]

\begin{prop}
\label{Proposition-connected-components}
The mapping $Sign \colon Compl_\Omega^{\pm} \to S_\Omega$ is
constant on every connected component of $Compl_\Omega$. 
\end{prop}
\begin{proof}
Since $n_0$ is the dimension of the kernel of $\Omega$, and $\Omega$ is fixed, $n_0$ is constant. Now the fact that $n_+$ and $n_-$ are locally constant follows from the fact that they are both lower semi-continuous (since being positive or negative are open conditions) and their sum is constant.
\end{proof}

Proposition \ref{Proposition-connected-components} allows us to define the induced mapping $\widetilde{Sign}\colon \pi_0Compl_\Omega^{\pm} 
\to S_\Omega$ on the set $ \pi_0Compl_\Omega^{\pm}$ of connected 
components of $Compl_\Omega^{\pm}$. 


We now formulate the following useful lemmas whose proofs will be given later in this section.

\begin{lem}
\label{Lemma-injectivity}
The group $G_\Omega^0$ acts transitively on each fiber of 
$Sign$. In particular, $\widetilde{Sign}$ is injective. 
\end{lem}

Lemma \ref{Lemma-injectivity} implies that the fibers of $Sign$ are
connected components of $Compl_\Omega^{\pm}$, each of which
is  a $G_\Omega^0$-orbit and is thus  
diffeomorphic to the quotient
$G_\Omega^0/G_{I,\Omega}$ for $I$ in this orbit. Thus, all $G_\Omega^0$-stabilizers of $I\in Sign^{-1}(n_+,n_-,n_0)$ are conjugate in $G_\Omega^0$, which
is the same as to say that all hermitian forms $h$ associated to $\Omega$
and $I \in Compl_\Omega^{\pm}$ of the same signature 
$(n_+,n_-,n_0)$ are $G_\Omega^0$-equivalent. 
We let 
$U(n_+,n_-,n_0)$ denote the stabilizer $G_{I,\Omega}$, well-defined up to conjugacy. 

\begin{lem}
\label{Lemma-surjectivity}
The mapping $\widetilde{Sign}\colon \pi_0 Compl_\Omega^{\pm}\to S_\Omega$ is surjective.
\end{lem}

Finally, we formulate the main result of this section.

\begin{thm} 
\label{Theorem-Compl-Omega}
The mapping $\widetilde{Sign}\colon \pi_0Compl_\Omega^{\pm} 
\to S_\Omega$ is a bijection, thus there are $2n+1-n_0$ connected components of $Compl_\Omega^{\pm}$ and they have the form 
$Sign^{-1}(n_+,n_-,n_0)$. 
Every connected component 
of $Compl_\Omega^{\pm}$ is a homogeneous manifold of the type $G^0_\Omega/U(n_+,n_-,n_0)$
and is a  smooth complex submanifold in $Compl$. 

If $n_0$ is even, there is precisely one connected component of $Compl_\Omega^{\pm}$ 
that contains twistor lines. This component is $Sign^{-1}\left(n-\frac{n_0}{2},n-\frac{n_0}{2},n_0\right)$ 
and is twistor path connected. 
If $n _0$ is odd, there are no connected components of $Compl_\Omega^{\pm}$
containing a twistor line.
\end{thm}

 Let $A=V_\RR/\Gamma$ be a torus with a complex structure $I \colon V_\RR \to V_\RR$.
\begin{cor}
\label{Corollary-Kahler}
 If the form $\Omega$ represents a K\"ahler class  in $H^{1,1}(A,\RR)$ then there are no
 twistor lines passing through $I$ in $Compl_\Omega$.
\end{cor}
Informally speaking, K\"ahler classes do not ``survive'' along twistor lines, this
is well known, for example, for twistor lines in the moduli space of $K3$ surfaces. Corollary \ref{Corollary-Kahler} is almost clear, since a form $\Omega$ represents  a K\"ahler class  in $H^{1,1}(A,\RR)$ if and only if
the form $h$ associated to $\Omega$ and $I$ is positive definite 
(see \cite[p. 303, {\it Riemann conditions 1}]{GrHarr} for a coordinate based exposition of this), that is, its signature is of the form $(2n,0,0)$ and hence, by Theorem \ref{Theorem-Compl-Omega}, the connected
component of $Compl_\Omega$ containing $I$ does not contain any twistor lines.

\begin{proof}[Proof of Lemma \ref{Lemma-injectivity}]
Let $I_1,I_2$ be periods in  $Compl_\Omega^{\pm}$. 
Assume that $Sign(I_1)=Sign(I_2)$, that is, the associated hermitian forms $h_1$ on $(V_\RR,I_1)$ and $h_2$
on $(V_\RR,I_2)$ have the same signature $(k,l,m)$.

First we fix orthogonal decompositions  into maximal positive, negative
and null subspaces  $V_\RR=V_+\oplus V_- \oplus V_0$ of $(V_\RR,I_1)$
with respect to $h_1$ and  
$V_\RR=W_+\oplus W_- \oplus V_0$ of $(V_\RR,I_2)$ with respect to $h_2$, here $\dim_\CC V_+=k=\dim_\CC W_+$, 
$\dim_\CC V_-=l=\dim_\CC W_-$ and $V_0$ is the null subspace for $\Omega$, $\dim_\CC V_0=m$. 
Let us choose an $h_1$-orthonormal basis $v_1,\dots,v_k$ of the complex subpace $(V_+,I_1)$ and an $h_2$-orthonormal basis $w_1,\dots,w_k$ of 
the complex subpace $(W_+,I_2)$, similarly choose 
$-h_1$-orthonormal and $-h_2$-orthonormal bases
for the subspaces $V_-,W_-$, and some (arbitrary) bases for 
$(V_0,I_1)$ and $(V_0,I_2)$.  
Then we define $g\in G=GL(V_\RR)$ by setting $g(v_1)=w_1,g(I_1v_1)=I_2w_1, g(v_2)=w_2,g(I_1v_2)=I_2w_2,\dots,g(v_k)=w_k,g(I_1v_k)=I_2w_k$,
and similarly for the remaining pairs of subspaces. Then we get $g \in G$ such that $g^*h_2=h_1$ and $I_2=gI_1g^{-1}$, which is equivalent to 
saying that $g^*\Omega=\Omega$
and $I_2=gI_1g^{-1}={}^g\!I_1$. This implies that $g\in G_\Omega$ and the assumption that both $I_1$ and $I_2={}^g\!I_1$
belong to the same connected component of $Compl$ tells us that  $g$
actually belongs to the connected subgroup component of $G$ and hence to $G_\Omega^0$.  
So $I_2$ belongs to the orbit of $I_1$ under the action of $G_\Omega^0$, and they belong to the same connected component of 
$Compl_\Omega^{\pm}$.
\end{proof}

\begin{proof}[Proof of Lemma \ref{Lemma-surjectivity}]
Let $(n_+,n_-,n_0)$ be a triple $\in S_\Omega$,  
$n_++n_-+n_0=2n$. 
Let us prove that there exists $I \in Compl_\Omega^{\pm}$ such that the hermitian form $h$ associated to $\Omega$ and $I$
has signature $(n_+,n_-,n_0)$. 
Choose a basis  $v_1,\dots,v_{4n}$ of $V_\RR$ such that the skew-symmetric matrix $\Omega(v_i,v_j)$
has $2\times 2$-block diagonal structure and set  $L_\Omega\colon V_\RR \to V_\RR$
to be the operator defined by this matrix in the chosen basis.
By the definition of $n_0$ there are $n_++n_-=2n-n_0$ nonzero $2\times 2$-blocks and $n_0$ zero $2\times 2$-blocks in the matix of $L_\Omega$. 
Let $V_\RR=P_1\oplus \dots \oplus P_{n_+} \oplus Q_1\oplus \dots \oplus Q_{n_-}\oplus U$ be a direct sum decomposition into $L_\Omega$-invariant
real 2-planes $P_j,Q_j$ and $2n_0$-dimensional real subspace $U=Ker\,L_\Omega$, where $L_\Omega|_{P_i},L_\Omega|_{Q_j}$ are  nonzero and $L_\Omega|_U=0$. 
Define $I\colon P_i\to P_i,1\leqslant i \leqslant n_+$, to be the appropriate negative multiple of $L_\Omega|_{P_i}$, 
$I\colon Q_j\to Q_j,1\leqslant j \leqslant n_-$, to be the appropriate positive multiple of $L_\Omega|_{Q_j}$, and $I\colon U \to U$
to be an arbitrary operator so as to form an operator $I \colon V_\RR \to V_\RR$ satisfying $I^2=-Id$. Then 
the decomposition $V_\RR=P_1\oplus \dots \oplus P_{n_+} \oplus Q_1\oplus \dots \oplus Q_{n_-}\oplus U$ is orthogonal with respect to $h(u,v)=\Omega(u,Iv)-i\Omega(u,v)$,
$h|_{P_i}>0$, $h|_{Q_j}<0$ and $h|_U=0$. Thus we have constructed the required $I$, this proves the surjectivity of $\widetilde{Sign}$. 
\end{proof}

\begin{proof}[Proof of Theorem \ref{Theorem-Compl-Omega}]
The part of the statement about the bijectivity of $\widetilde{Sign}$
and the $G_\Omega^0$-orbit structure of the connected components
of $Compl_\Omega^{\pm}$ follows from Lemma \ref{Lemma-injectivity}
and Lemma \ref{Lemma-surjectivity}. 

\medskip

{\bf The complex manifold structure.}
In order to show that every connected component of $Compl_\Omega^{\pm}$
is a complex submanifold in $Compl$ we need to identify the tangent
space $T_ICompl_\Omega$ as a complex subspace in $T_ICompl$ for every $I \in Compl_\Omega$. We recall that the complex structure operator $l_I\colon T_ICompl \to T_ICompl$  is given by left
multiplication by $I$. 
By definition $G_\Omega=\{g\in G\,|\,{}^t\!g \Omega g =\Omega\}\subset G$, so
$$T_eG_\Omega \cong \{Y \in T_eG\,|\, {}^t\!Y\Omega+\Omega Y=0\}.$$  
The stabilizer of $I$ in $G_\Omega$ is  $G_{I,G}=G_I \cap G_\Omega$,
thus $T_ICompl_\Omega \cong T_eG_\Omega/T_eG_{I,\Omega}$.

For $X \in T_eG_\Omega$, consider the decomposition $X=\frac{1}{2}(X+X^I)+\frac{1}{2}(X-X^I)$ into $I$-commuting and $I$-anticommuting components. Then
$\frac{1}{2}(X+X^I)$ and $\frac{1}{2}(X-X^I)$ satisfy the equality ${}^t\!Y\Omega+\Omega Y=0$,  
so that $\frac{1}{2}(X+X^I) \in T_eG_{I,\Omega}$ and thus the isomorphism $T_ICompl_\Omega \cong T_eG_\Omega/T_eG_{I,\Omega}$  implies the natural isomorphism
\begin{equation}
\label{T-I-Compl-Omega}
T_ICompl_\Omega \cong \{Y \in T_eG\,|\,
YI=-IY,\,{}^t\!Y\Omega+\Omega Y=0\}.
\end{equation}  
Thus $T_ICompl_\Omega$ is obviously $l_I$-invariant, so it is a complex subspace in $T_ICompl$ for every $I\in Compl_\Omega$, which shows that every connected component of $Compl_\Omega$, or, which is the same, 
of $Compl_\Omega^{\pm}$, is a complex submanifold in $Compl$. 
\medskip

{\bf The component containing a twistor line.} 
We first note that $Compl_\Omega$ contains  a twistor line $S=S(I,J)$
if and only if $\Omega$ is both $I$- and $J$-invariant, that is,
$I,J \in Compl_\Omega$. 

Let us assume that $\Omega$ is invariant with respect to anticommuting
complex structure operators $I$ and $J$,
that is $\Omega(Iu,Iv)=\Omega(Ju,Jv)=\Omega(u,v)$ for all $u,v\in V_\RR$. 
Then for all $u,v\in V_\RR$ the form $h$ associated to $\Omega$ and $I$ satisfies  
$$h(Ju,Jv)\overset{def}{=}\Omega(Ju,IJv)-i\Omega(Ju,Jv)=
-\Omega(Ju,JIv)-i\Omega(Ju,Jv)=$$
$$=-\Omega(u,Iv)-i\Omega(u,v)=-\overline{h(u,v)}.$$
This implies that the hermitian forms $h(u,v)$ and $-\overline{h(u,v)}$ are equivalent under $J$, therefore they must have the same signature, 
so that the signature $(n_+,n_-,n_0)$ of $h(u,v)$  and the signature $(n_-,n_+,n_0)$ of $-\overline{h(u,v)}$ 
are equal, that is, $n_+=n_-$. It also follows that $n_0$ is even.
 
Conversely, if, for the hermitian form $h$ on $(V_\RR,I)\times (V_\RR,I)$ associated to $\Omega$ and $I$, we have the equality $n_+=n_-$, 
we construct a complex structure operator $J\colon V_\RR \to V_\RR$ anticommuting with $I$
and leaving $\Omega$ invariant as follows. 
Fix an $h$-orthogonal decomposition $V_\RR=V_+\oplus V_-\oplus V_0$ , where
$V_+$, $V_-$, 
$V_0$ are maximal positive, negative and null subspaces of $h$, 
which are complex subspaces of $(V_\RR,I)$,
and put $k:=\dim_\CC V_+=n_+=n_-=\dim_\CC V_-,
l:=n_0=\dim_\CC V_0$.
 Note that since $n_+ +n_- + n_0=2k+l=2n$, $l=n_0$ is an even number. 

Let $u_1,\dots, u_{k}$ and 
 $v_1,\dots, v_{k}$ be $h$-orthonormal and, respectively, $-h$-orthonormal bases of the complex subspaces
 $(V_+,I)$ and $(V_-,I)$.
  Let $w_1,\dots, w_{l}$ be any basis of the complex subspace $(V_0,I)$.  
Define $J\colon V_\RR \to V_\RR$ by setting 
$$J(u_1)=v_1, J(Iu_1)=-Iv_1, \dots, J(u_k)=v_k, J(Iu_k)=-Iv_k,$$
$$J(v_1)=-u_1, J(Iv_1)=Iu_1, \dots, J(v_k)=-u_k, J(Iv_k)=Iu_k$$
and, using that $l$ is an even number, by setting
$$J(w_1)=w_2, J(Iw_1)=-Iw_2, J(w_2)=-w_1, J(Iw_2)=Iw_1, \dots$$
$$ J(w_{l-1})=w_l, J(Iw_{l-1})=-Iw_l, J(w_l)=-w_{l-1}, J(Iw_l)=Iw_{l-1}.$$

Then $J$ anticommutes with $I$ and one easily verifies that $J$ takes $h$ to $-\overline{h}$,
which, by the definition of $h$, means that  
$$\Omega(Ju,IJv)-i\Omega(Ju,Jv) = h(Ju,Jv) =-\overline{h(u,v)} =
-\Omega(u,Iv)-i\Omega(u,v),$$
so that $\Omega(Ju,Jv)=\Omega(u,v)$, that is, $\Omega$  is $J$-invariant.

\medskip

{\bf Twistor-path connectivity of $Sign^{-1}(n-\frac{n_0}{2},n-\frac{n_0}{2},n_0)$.} 
We follow the general lines in
Section \ref{subsecPhi}. Let $S=S(I,J)$ be a twistor line in $Sign^{-1}(n-\frac{n_0}{2},n-\frac{n_0}{2},n_0)$.

Consider the subgroups $G_{J,\Omega}=G_J \cap G_\Omega,
G_{K,\Omega}=G_K \cap G_\Omega$, 
$G_{\mathbb{H},\Omega}=G_{J,\Omega}\cap G_I=G_{K,\Omega}\cap G_I$ and define the mapping
$$\Phi\colon G_{J,\Omega} \times G_{K,\Omega} \to Compl_\Omega^{\pm},$$
$$(g_1,g_2)\mapsto {}^{g_1g_2}\!I \in Compl_\Omega^{\pm}.$$

In order to proceed as in Section \ref{subsecPhi} we need to show that
$d\Phi_{(e,e)}$ is surjective. 
Again, as in the proof of Proposition
\ref{Injectivity-proposition},
 one can show that $Ker\,d\Phi_{(e,e)}=T_eG_{\mathbb{H},\Omega}\times T_eG_{\mathbb{H},\Omega}$,
so that $$d\Phi_{(e,e)}(T_eG_\Omega)\cong T_eG_{J,\Omega}/T_eG_{\mathbb{H},\Omega} \times T_eG_{K,\Omega}/T_eG_{\mathbb{H},\Omega}
\subset T_ICompl_\Omega.$$


To show the surjectivity of $d\Phi_{(e,e)}$ we need another inclusion $$T_ICompl_\Omega\subset T_eG_{J,\Omega}/T_eG_{\mathbb{H},\Omega} \times T_eG_{K,\Omega}/T_eG_{\mathbb{H},\Omega}.$$ 
For this we use the identification in Equation (\ref{T-I-Compl-Omega}) 
and first note that, for $Z\in T_ICompl_\Omega$
($Z$ anticommutes with $I$) the vector $Z+Z^J$ commutes with $J$
and anticommutes with $I$, hence anticommutes with $K$ as well. 
Next, by $J$-invariance of $\Omega$ we have that 
$Z+Z^J$ satisfies the relation with $\Omega$ in (\ref{T-I-Compl-Omega}), hence itself belongs to 
$$\{Y\in T_ICompl_\Omega\,|\, YJ=JY\}=T_eG_{J,\Omega}/T_eG_{\mathbb{H},\Omega}.$$ 
Now $Z-\frac{1}{2}(Z+Z^J)=\frac{1}{2}(Z-Z^J)$ commutes with $K$,
so that, due to the fact that $Z+Z^J$ anticommutes with $K$, 
 we have that $\frac{1}{2}(Z-Z^J)$ is equal to $\frac{1}{2}(Z+Z^K)$. 
Again, by $K$-invariance of $\Omega$ we have  $Z+Z^K \in T_eG_{K,\Omega}/T_eG_{\mathbb{H},\Omega}$,
and so  we may write 
$$Z=\frac{1}{2}(Z+Z^J)+\frac{1}{2}(Z+Z^K) \in T_eG_{J,\Omega}/T_eG_{\mathbb{H},\Omega}\times T_eG_{K,\Omega}/T_eG_{\mathbb{H},\Omega}.$$ 
Thus we have shown the inclusion $$T_ICompl_\Omega \subset T_eG_{J,\Omega}/T_eG_{\mathbb{H},\Omega}\times T_eG_{K,\Omega}/T_eG_{\mathbb{H},\Omega}$$
and we may finally write 

$$
T_ICompl_\Omega = T_eG_{J,\Omega}/T_eG_{\mathbb{H},\Omega}\times T_eG_{K,\Omega}/T_eG_{\mathbb{H},\Omega},
$$
which proves the surjectivity of our $d\Phi_{(e,e)}$. 
Now we can directly 
proceed with the ``three lines argument'' as in the proof of Theorem \ref{Theorem-connectivity}
to establish twistor path connectivity of  the 
connected component $Sign^{-1}(n-\frac{n_0}{2},n-\frac{n_0}{2},n_0)$ of $Compl_\Omega^{\pm}$. 
\end{proof}

\begin{rem}
In the proof of Theorem \ref{Theorem-Compl-Omega} we used
a particular choice of points $J,K \in S=S(I,J)$ around which we rotate
our twistor line $S$. In fact, we could use any choice of $I_1,I_2\in S$
such that $I,I_1,I_2$ are linearly independent in $End\,V_{\mathbb{R}}$.
Indeed, defining in a natural way the mapping
$$\Phi\colon G_{I_1,\Omega}\times G_{I_2,\Omega} \to Compl_\Omega^{\pm},$$
$$(g_1,g_2)\mapsto {}^{g_1g_2}\!I,$$
we get, as usually, that $Im\,d\Phi_{(e,e)}\cong 
T_eG_{I_1,\Omega}/T_eG_{\mathbb{H},\Omega}\times G_{I_2,\Omega}/T_eG_{\mathbb{H},\Omega} \subset T_ICompl_\Omega \cong 
T_eG_{J,\Omega}/T_eG_{\mathbb{H},\Omega}\times G_{K,\Omega}/T_eG_{\mathbb{H},\Omega}$, the last isomorphism was proved in 
Theorem \ref{Theorem-Compl-Omega}. Now we define the mapping
$$T_ICompl_\Omega \cong T_eG_{J,\Omega}/T_eG_{\mathbb{H},\Omega}\times G_{K,\Omega}/T_eG_{\mathbb{H},\Omega} \to T_eG_{I_1,\Omega}/T_eG_{\mathbb{H},\Omega}\times T_eG_{I_2,\Omega}/T_eG_{\mathbb{H},\Omega},$$
$$Y\mapsto (Y+Y^{I_1},Y+Y^{I_2}).$$
As $I_1,I_2 \in Compl_\Omega$, this mapping  is indeed correctly defined. 
Again, conveniently identifying $$T_eG_{J,\Omega}/T_eG_{\mathbb{H},\Omega}
\cong \{Z \in End\,V_\mathbb{R}\,|\, ZI=-IZ,ZJ=JZ, Z^t\Omega=\Omega Z \},$$ and similarly for $T_eG_{K,\Omega}/T_eG_{\mathbb{H},\Omega}
$, 
it is not hard to verify that this mapping has zero kernel. Note that for general $I_1,I_2$ we do not have that the image of a vector $Y$
anticommuting with $I$ under our mapping is again a vector anticommuting with $I$, so
in the verification we do not identify the target $T_eG_{I_1,\Omega}/T_eG_{\mathbb{H},\Omega}\times T_eG_{I_2,\Omega}/T_eG_{\mathbb{H},\Omega}$ with $I$-anticommuting vectors and consider it just as it is. Thus
$$T_ICompl_\Omega \cong T_eG_{I_1,\Omega}/T_eG_{\mathbb{H},\Omega}\times T_eG_{I_2,\Omega}/T_eG_{\mathbb{H},\Omega}$$
and so $d\Phi_{(e,e)}$ is surjective, allowing us to proceed with the proof of connectivity in a usual way.
\end{rem}

\end{document}